\documentclass[reqno]{amsart}
\usepackage{amssymb,amsthm,amsmath,amsxtra}

\usepackage{amssymb,amsthm,amsfonts,amstext}
\usepackage{amsmath}

\newcommand{\pfrac}[2]{\genfrac{}{}{}{1}{#1}{#2}}

\newtheorem{theorem}{Theorem}

\newtheorem{example}[theorem]{Example}
\newtheorem{lemma}[theorem]{Lemma}

\newtheorem{proposition}[theorem]{Proposition}
\newtheorem{remark}[theorem]{Remark}

\begin{document}

\title[Large Deviations via Aubry-Mather theory]{Large Deviations  for stationary probabilities of a family of continuous  time Markov chains via Aubry-Mather theory}

\author{Artur O. Lopes}
\address{UFRGS, Instituto de Matem\'atica, Av. Bento Gon\c calves, 9500. CEP 91509-900, Porto Alegre, Brasil}
\curraddr{}
\email{arturoscar.lopes@gmail.com}
\thanks{}

\author{Adriana Neumann}
\address{UFRGS, Instituto de Matem\'atica, Av. Bento Gon\c calves, 9500. CEP 91509-900, Porto Alegre, Brasil}
\curraddr{}
\email{aneumann@impa.br}
\thanks{}

\date{\today}
\maketitle

\bigskip

\begin{abstract}

In the present paper, we consider  a family of continuous time symmetric random walks indexed by $k\in \mathbb{N}$, $\{X_k(t),\,t\geq 0\}$. For each $k\in \mathbb{N}$ the
matching random walk  take values in the finite set of states $\Gamma_k=\frac{1}{k}(\mathbb{Z}/k\mathbb{Z})$; notice that $\Gamma_k$ is a subset of $\mathbb{S}^1$, where $\mathbb{S}^1$ is
the unitary circle. The infinitesimal generator of such chain is denoted by $L_k$.
The stationary probability for such process converges to the uniform distribution on the circle, when $k\to \infty$.
Here we want to study other natural measures, obtained via a limit on $k\to \infty$, that are concentrated on some points of $\mathbb{S}^1$.
We will disturb this process by a potential and study for each $k$ the perturbed stationary measures of this new process when $k\to \infty$.

We disturb the system considering  a fixed $C^2$ potential
$V: \mathbb{S}^1 \to \mathbb{R}$ and  we will denote by $V_k$ the restriction of
$V$ to $\Gamma_k$.
Then, we define
a non-stochastic semigroup generated by the matrix
$k\,\, L_k + k\,\, V_k$, where $k\,\, L_k $ is the infinifesimal generator of $\{X_k(t),\,t\geq 0\}$. From the continuous time Perron's Theorem one can normalized such semigroup, and, then we get another stochastic semigroup which
generates a continuous time Markov Chain taking values on $\Gamma_k$. This new chain is called the continuous time Gibbs state associated to the potential $k\,V_k$, see \cite{LNT}. The stationary probability vector for such Markov Chain
is denoted by $\pi_{k,V}$. We assume that the maximum of $V$ is attained in a unique point $x_0$ of $\mathbb{S}^1$, and from this will follow that $\pi_{k,V}\to \delta_{x_0}$. Thus, here, our main goal is to analyze the large deviation
principle for the family $\pi_{k,V}$, when $k \to\infty$.  The deviation function $I^V$, which is defined on
$ \mathbb{S}^1$, will be obtained from  a procedure based on  fixed points of the Lax-Oleinik operator and Aubry-Mather theory.
In order to obtain the associated Lax-Oleinik operator we use the Varadhan's Lemma for the process  $\{X_k(t),\,t\geq0\}$. For a careful analysis of the problem we present full details of the proof of the   Large Deviation Principle, in the Skorohod space, for such family of Markov Chains, when $k\to \infty$.
Finally, we compute the entropy of the invariant probabilities on the  Skorohod space associated to the Markov Chains we analyze.

\end{abstract}

\bigskip

\section{Introduction}\label{sec1}

We will study  a family of continuous time Markov Chains indexed by $k\in \mathbb{N}$, for each $k\in \mathbb{N}$ the corresponding Markov Chain  take values in the finite set of states $\Gamma_k=\frac{1}{k}(\mathbb{Z}/k\mathbb{Z})$.
Let $\mathbb{S}^1$ be the unitary circle which can be identified with the interval $[0,1)$. In this way we identify $\Gamma_k$ with $\{0,\,1/k,\, 2/k,...,\,(k-1)/k\}$ in order to simplify the notation.
We will analyse below a limit procedure on $k\to \infty $ and this is the reason why we will consider that the values of the states of the chain are in the unitary circle. The continuous time Markov Chain with index $k$ has the following behaviour: if the particle is at
$j/k$ it waits an exponential time of parameter $2$ and then jumps either to $(j-1)/k$ or to $(j+1)/k$ with probability $1/2$. In order to simplify the notation, we omit the indication that the the sum $j+1$ is mod $k$ and the same for the subtraction $j-1$; we will do this without other comments in the rest of the  text.
The skeleton of this continuous time Markov Chain has matrix of transitions $\mathcal{P}_k=(p_{i,j})_{i,j}$  such that the element $p_{j,j+1}$ describes the probability of transition  of $i/k$ to  $j/k$, which is $p_{i,i+1}=p_{i,i-1}=1/2$ and $p_{i,j}=0$,  for all $j\neq i$. The infinitesimal generator is the matrix
$L_k=2(\mathcal{P}_k-I_k)$, where $I_k$ is the identity matrix, in words  $L_k$ is a  matrix  that is equal to $-2$ in the diagonal $L_{i,j}=1$ above and below the diagonal, and the rest is zero.
Notice that $L_k$ is symmetric matrix.
For instance, take $k=4$,
$$L_4= \left(
\begin{array}{cccc}
-2 & 1 & 0 & 1 \\
1 & -2 & 1 & 0 \\
0 & 1 & -2 & 1 \\
1 & 0 & 1 & -2 \\
\end{array}\right).
$$

We can write this infinitesimal generator as an operator acting on functions $f: \Gamma_k\to\mathbb{R}$ as
\begin{equation}\label{ger}
\begin{split}
(\mathcal{L}_kf)(\pfrac{j}{k})=\big[f(\pfrac{j+1}{k})-f(\pfrac{j}{k})\big]+\big[f(\pfrac{j-1}{k})-f(\pfrac{j}{k})\big].
\end{split}
\end{equation}
Notice that this expression describes the  infinitesimal generator of continuous time random walk.
For each $k\in \mathbb{N}$, we denote $P_k(t)=e^{t\, L_k}$ the semigroup associated to this infinitesimal generator.
We also denote by $\pi_k$ the uniform probability on $\Gamma_k$. This is the invariant probability for the above defined continuous Markov Chain.
The probability $\pi_k$ converges to the Lebesgue
measure on $ \mathbb{S}^1$, as  $k \to \infty$.

Fix $T>0$ and $x_0\in\mathbb{S}^1$, let $\mathbb{P}_k$ be probability on the Skorohod space $D[0,T]$, the space of \emph{c\`adl\`ag} trajectories taking values on $\mathbb{S}^1$, which are induced by  the infinitesimal generator  $k\mathcal L_k$ and the initial probability $\delta_{x_k(x_0)}$, which is the Delta of Dirac at $x_k(x_0):=\lfloor k x_0\rfloor/k\in \Gamma_k$, where $x_k(x_0)$ is the closest point to $x_0$ on the left of $x_0$ in the set $\Gamma_k$. Denote by $\mathbb{E}_k$ the expectation with respect to $\mathbb{P}_k$ and by $\{X_k(t)\}_{t\in [0,T]}$ the continuous time Markov chain with the infinitesimal generator $k\mathcal L_k$. One of our goals is described in the Section \ref{sec2} which is to establish a Large Deviation Principle for $\{\mathbb{P}_k\}_k$ in $D[0,T]$. This will be used later on the Subsection \ref{subsec3.1} to define the Lax-Oleinik semigroup. One can ask: why we use this time scale? Since  the continuous time symmetric random walk converges just when the time is rescaled with speed $k^2$, then taking speed $k$ the symmetric random walk  converges to a constant trajectory. Here the setting follows similar ideas  as the ones in the papers \cite{A1} and \cite{A2}, where N. Anantharaman used the Shilder's Theorem. The Shilder's Theorem says that for $\{B_t\}_t$  (the standard Brownian Motion) the sequence $\{\sqrt{\varepsilon}B_t\}_t$, which  converges to a trajectory constant equal to zero, when $\varepsilon\to 0$, has   rate of  convergence equal to $I(\gamma)=\int_0^T\frac{(\gamma'(s))^2}{2}\,ds$, if $\gamma:[0,T]\to\mathbb{R}$ is absolutely continuous, and $I(\gamma)=\infty$, otherwise.

We proved that the sequence of measures $\{\mathbb{P}_k\}_k$ satisfy the large deviation principle with  rate  function $I_T: D[0,T]\to \mathbb{R}$ such that
\begin{equation*}
\begin{split}
&I_{T}(\gamma)=\int_0^T
\Big\{\gamma'(s)\log\Big(\frac{\gamma'(s)+\sqrt{(\gamma'(s))^2+4}}{2}\Big)
-\sqrt{(\gamma'(s))^2+4}+2\Big\}\,ds,
\end{split}
\end{equation*}
if $\gamma \in \mathcal{AC}[0,T]$ and $I_{T}(\gamma)=\infty$, otherwise.

Finally, in Section \ref{sec3}, we consider this system disturbed by a $C^2$ potential $V:\mathbb S^1 \to \mathbb{R}.$ The restriction of $V$  to $ \Gamma_k$ is denoted by $V_k$. From the continuous  time Perron's Theorem
we get an eigenvalue and an eigenfunction for the operator $k\,L_k + k \, V_k$. Then, normalizing the semigroup associated to $k\,L_k + k \, V_k$ via the eigenvalue and eigenfunction of this operator, we obtain a new continuous time Markov Chain, which is  called the Gibbs Markov Chain associated to $k\, V_k$  (see \cite{BEL} and \cite{LNT}). Denote by $\pi_{k,V}$ the initial stationary vector of this family of continuous time Markov Chains indexed by $k$ and which takes values on  $ \Gamma_k\subset \mathbb S^1$. We investigate the large deviation properties of this family of stationary vectors which are probabilities on $\mathbb S^1$, when $k\to \infty$.
More explicitly, roughly speaking, the deviation function $I^V$ should satisfy the property: given an interval $[a,b]$
$$\lim_{k\to \infty} \pfrac{1}{k} \,\log \,\pi_{k,V}\,[a,b]\,=\,-\inf_{x \in [a,b]}{I^V(x)}.$$

If $V:\mathbb S^1 \to \mathbb{R}$ attains the maximal value in just one  point $x_0$, then, $\pi_{k,V}$ weakly converge, as $k\to \infty$, to the delta Dirac in $x_0.$
We will use results of Aubry-Mather theory (see \cite{BG}, \cite{CI},  \cite{Fath}  or \cite{Fat}) in order to exhibit the deviation function $I^V$, when $k \to \infty$.

It will be natural to consider the Lagrangian defined on $S^1$ given by
$$
L(x,v)= - V(x) + v \log(  (v + \sqrt{v^2 + 4} )/2 ) - \sqrt{v^2
+ 4} + 2,
$$
which is convex and superlinear. It is easy  to get the explicit expression of the associated Hamiltonian,

As we will see the deviation function is obtained from certain weak KAM solutions of the associated Hamilton-Jacobi equation (see Section 4 and 7 in \cite{Fat}). In the one-dimensional case $\mathbb S^1$ the weak KAM solution can be in some cases explicitly obtained (for instance when $V$ as a unique point of maximum). From the conservation of energy (see \cite{Car}), in this case,  one can get a solution (periodic) with just one point of lack of differentiability.

It follows from the continuous time Perron's Theorem that the probability vector $\pi_{k,V}$ depends for each $k$ on  a left eigenvalue
and on a right eigenvalue. In this way, in the limit procedure, this will require in our reasoning  the use of the positive time and negative time Lax-Oleinik operators (see \cite{Fat}).

From a theoretical perspective, following our reasoning, one can think that we are looking for the maximum of a function
$V:\mathbb S^1 \to \mathbb{R}$ via an stochastic procedure based on  continuous time Markov Chains  taking
values on the finite lattice $\Gamma_k$, $k \in \mathbb{N}$, which is a discretization of the circle $\mathbb S^1$.
Maybe this can be explored as an alternative approach to Metropolis algorithm, which is base in frozen arguments. In our setting the
deviation function $I^V$ gives bounds for the decay of the probability that the stochastic procedure corresponding to a certain $k$ does not
localize the maximal value.

Moreover, in the  Section \ref{sec4} we compute explicitly the entropy of the Gibbs state on the Skhorod space associated to the potential  $k\,V_k$. In this moment we  need to generalize a  result which was obtained in \cite{LNT}. After that, we take the limit on $k\to \infty$, and we obtain the entropy for the  limit process which in this case is shown to be zero.

\section{Large Deviations on the Skorohod space \\for the unperturbed system}\label{sec2}
The goal of this section is to prove the Large Deviation Principle for the sequence of measures
 $\{\mathbb{P}_k\}_k$ on $D[0,T]$, defined in Section \ref{sec1}. We recall that $\mathbb{P}_k$ is induced by the continuous time random walk, which has infinitesimal generator $k\mathcal L_k$, see \eqref{ger}, and the initial measure $\delta_{x_k(x_0)}$, which is the Delta of Dirac at $x_k(x_0)=\lfloor k x_0\rfloor/k\in \Gamma_k$.
\begin{theorem}\label{teo1}
The sequence of probabilities $\{\mathbb{P}_k\}_k$ satisfies:
\begin{itemize}
\item[] Upper Bound: For all $\mathcal{C}\subset D[0,T]$ closet set,
 \begin{equation*}
\begin{split}
&\varlimsup_{k\to\infty}\frac{1}{k}\log\mathbb{P}_k\Big[X_k\in \mathcal C\Big]\leq -\inf_{\gamma\in \mathcal O}I_{T}(\gamma) .
\end{split}
\end{equation*}
\item[] Lower Bound: For all $\mathcal{O}\subset D[0,T]$ open set,\begin{equation*}
\begin{split}
&\varliminf_{k\to\infty}\frac{1}{k}\log\mathbb{P}_k\Big[X_k\in \mathcal O\Big]\geq -\inf_{\gamma\in \mathcal O}I_{T}(\gamma) .
\end{split}
\end{equation*}
\end{itemize}
The rate function $I_{T}: D[0,T]\to \mathbb{R}$ is
\begin{equation}\label{funcional}
\begin{split}
&I_{T}(\gamma)=\int_0^T
\Big\{\gamma'(s)\log\Big(\frac{\gamma'(s)+\sqrt{(\gamma'(s))^2+4}}{2}\Big)
-\sqrt{(\gamma'(s))^2+4}+2\Big\}\,ds,
\end{split}
\end{equation}
if $\gamma \in \mathcal{AC}[0,T]$ and $I_{T}(\gamma)=\infty$, otherwise.
\end{theorem}
The set $\mathcal{AC}[0,T]$ is  the set of all absolutely continuous functions $\gamma:[0,T] \to  \mathbb{S}^1$.
Saying that a function $\gamma:[0,T] \to  \mathbb{S}^1$ is absolutely continuous means that for
all $\varepsilon>0$ there is $\delta>0$, such that, for all family of intervals $\{(s_i,t_i)\}_{i=1}^{n}$ on $[0,T]$, with $\sum_{i=1}^{n} t_i-s_i<\delta$, we have
$\sum_{i=1}^{n} \gamma(t_i)-\gamma(s_i)<\varepsilon$.

\begin{proof}
This proof is divided in two parts: upper bound and lower bound.
The proof of the upper bound is on Subsections \ref{subsec2.2} and  \ref{subsec2.3}.
And, the proof of the lower bound is Subsection  \ref{subsec2.4}.
In the Subsection  \ref{subsec2.1}, we prove some useful tools for this proof, like the one related to the perturbation of the system and also the computation of the Lengendre transform.
\end{proof}

\subsection{Useful tools}\label{subsec2.1}
In this subsection we will prove some important results for the upper bound and for the lower bound. More specifically, we will study  a typical pertubation of the original system and also the Radon-Nikodym derivative of this process. Moreover, we will compute the Fenchel-Legendre transform for a function $H$ that appears in a natural way in the Radon-Nikodym derivative.

For a time partition $0=t_0<t_1<t_2<\dots<t_n=T$ and  for $\lambda_i:[t_{i-1},t_i]\to\mathbb{R}$ a linear function with linear coefficient $\lambda_i$, for $i\in\{1,\dots,n\}$, consider a polygonal function $\lambda:[0,T]\to\mathbb R$  as $\lambda(s)=\lambda_i(s)$  in $[t_{i-1},t_i]$, for all $i\in\{1,\dots,n\}$.

For each $k\in \mathbb N$ and for the polygonal function $\lambda:[0,T]\to \mathbb R$, defined above, consider  the martingale
\begin{equation}\label{martingale1}
M^k_t=\exp\Big\{k\,\big[\lambda(t)X_k(t)-\lambda(0)X_k(0)-\frac{1}{k}\int_0^t e^{-k\lambda(s) X_k(s)}(\partial_s+k{\mathcal L}_k) e^{k\lambda(s) X_k(s)}ds\big]\Big\},
\end{equation}
notice that $M^k_t$ is positive and $\mathbb{E}_k[M^k_t]=1$, for all $t\geq 0$, see Appendix 1.7 of \cite{KL}. Making a simple calculation, the part of the expression inside the integral can rewritten as
\begin{equation*}
\begin{split}
 e^{-k\lambda(s)X_k(s)}k{\mathcal L}_k e^{k\lambda(s)X_k(s)}=&
 \,e^{-k\lambda(s)X_k(s)}k\Big\{e^{k\lambda(s)( X_k(s)+1/k)}-e^{k\lambda(s) X_k(s)}\\&\qquad\qquad\qquad+e^{k\lambda(s)( X_k(s)-1/k)}-e^{k\lambda(s) X_k(s)}\\
 =&\,e^{-k\lambda(s)X_k(s)}k \,e^{k\lambda(s) X_k(s)}\Big\{e^{\lambda(s)}-1+e^{-\lambda(s)}-1\Big\}\\
 =&k \,\Big\{e^{\lambda(s)}+e^{-\lambda(s)}-2\Big\}\\
  =&k \,H(\lambda(s)),\\
 \end{split}
\end{equation*}
where $H(\lambda):=e^\lambda+e^{-\lambda}-2$. Since $\lambda$ is a polygonal function,
the other part of the expression inside the integral is equal to
\begin{equation*}
\begin{split}
 e^{-k\lambda(s)X_k(s)}\partial_s\,e^{k\lambda(s)X_k(s)}=&
 \,e^{-k\lambda(s)X_k(s)}\,e^{k\lambda(s) X_k(s)}k\lambda'(s)X_k(s)\\
 =&
 \,k\lambda'(s)\,X_k(s)=\,k\sum_{i=0}^{n-1}\lambda_{i+1}\textbf{1}_{[t_{i},t_{i+1}]}(s)\,X_k(s).
 \end{split}
\end{equation*}
Using telescopic sum, we have
\begin{equation*}
\begin{split}\lambda(T)X_k(T)-\lambda(0)X_k(0)&=\sum_{i=0}^{n-1}\big[\lambda_{i+1}(t_{i+1})X_k(t_{i+1})-\lambda_{i}(t_{i})X_k(t_{i})\big]\\&=\sum_{i=0}^{n-1}\big[\lambda_{i+1}(t_{i+1})X_k(t_{i+1})-\lambda_{i+1}(t_{i})X_k(t_{i})\big].
 \end{split}
\end{equation*}
The last equality follows from the fact that $\lambda$ is a polygonal function $(\lambda_{i}(t_{i})=\lambda_{i+1}(t_{i}))$.
Thus, the martingale $M^k_T$ becomes
\begin{equation}\label{martingale2}
\begin{split}
M^k_T=\exp\Bigg\{k\,\sum_{i=0}^{n-1}\Big[&\,\lambda_{i+1}(t_{i+1})X_k(t_{i+1})-\lambda_{i+1}(t_{i})X_k(t_{i})\\&-\int_{t_{i}}^{t_{i+1}} \!\!\![\,\lambda_{i+1}\,X_k(s)+ H(\lambda_{i+1}(s))\,]\,ds \,\Big]\Bigg\}.
\end{split}
\end{equation}

\begin{remark}\label{lambda_dif}
If $\lambda:[0,T]\to\mathbb R$ is an absolutely continuous function, the expression for the martingale $M^k_T$ can be rewritten as
\begin{equation*}
\begin{split}
M^k_T=\exp\Bigg\{k\,\Big[\lambda(T)X_k(T)-\lambda(0)X_k(0)-\int_{0}^{T} \![\,\lambda'(s)\,X_k(s)+ H(\lambda(s))\,]\,ds \,\Big]\Bigg\}.
\end{split}
\end{equation*}
\end{remark}

\bigskip

Define a measure on  $D[0,T]$ as
\begin{equation*}
\mathbb{P}_k^\lambda[A]=\mathbb{E}_k[\mathbf{1}_A(X_k)\,M^k_T],
\end{equation*}
for all set $A$ in $D[0,T]$.  For us $\mathbf{1}_A$ is the indicator function of the set $A$, it means that $\mathbf{1}_A(x)=1$ if $x\in A$ or $\mathbf{1}_A(x)=0$ if $x\notin A$.

One can observe that this measure is associated to  a non-homogeneous in time process, which have infinitesimal generator acting on functions $f: \Gamma_k\to\mathbb{R}$ as
\begin{equation*}
\begin{split}
(\mathcal{L}_k^{\lambda(t)} f)(\pfrac{j}{k})=e^{ \lambda(t) }\big[f(\pfrac{j+1}{k})-f(\pfrac{j}{k})\big]+e^{-\lambda(t) }\big[f(\pfrac{j-1}{k})-f(\pfrac{j}{k})\big].
\end{split}
\end{equation*}
By Proposition 7.3 on Appendix 1.7 of \cite{KL}, $M^k_T$ is a Radon-Nikodym derivative   $\frac{d\mathbb{P}_k^\lambda}{d\mathbb{P}_k}$.

To finish this section, we will analyse the properties of the function $H$, which appeared in the definition of the martingale $M_T^k$.

\begin{lemma}\label{Legendre}
Consider the function
\begin{equation*}
\begin{split}
H(\lambda)=e^\lambda+e^{-\lambda}-2
\end{split}
\end{equation*}
 the Fenchel-Legendre transform of $H$ is
\begin{equation}\label{legendre1}
\begin{split}
L(v)=\sup_{\lambda} \big\{\lambda v -H(\lambda)\big\}=v\log\Big(\pfrac{1}{2}\Big(v+\sqrt{(v)^2+4}\Big)\Big)-\sqrt{(v)^2+4}+2\,.
\end{split}
\end{equation}
Moreover, the supremum above is attain on $\lambda_v=\log\Big(\pfrac{1}{2}\Big(v+\sqrt{(v)^2+4}\Big)\Big)$.
\end{lemma}
\begin{proof}
Maximizing  $\lambda v -(e^\lambda+e^{-\lambda}-2)$ on $\lambda$, we obtain the expression on \eqref{legendre1}.
\end{proof}

Then, we can rewrite the rate functional $I_{T}: D[0,T]\to \mathbb{R}$, defined in \eqref{funcional}, as
\begin{eqnarray}\label{func_melhor}
I_{T}(\gamma)=\left\{\begin{array}{ll}\int_0^T L(\gamma' (s))\,ds,& if\,\, \gamma \in \mathcal{AC}[0,T],\\
\infty,&otherwise.
\end{array}\right.
\end{eqnarray}

\subsection{Upper bound for compact sets}\label{subsec2.2}
Let $\mathcal C$ be an open set of $D[0,T]$. For all $\lambda:[0,T]\to\mathbb R$ polygonal function as in Subsection \ref{subsec2.1}, we have
\begin{equation*}
\begin{split}
&\mathbb{P}_k\Big[X_k\in \mathcal C\Big]
=\mathbb{E}_k^\lambda\Big[\mathbf{1}_{\mathcal C}(X_k^\lambda)\frac{d\mathbb{P}_k}{d\mathbb{P}_k^\lambda}\Big]=
\mathbb{E}_k^\lambda\Big[\mathbf{1}_{\mathcal C}(X_k^\lambda)(M^k_T)^{-1}\Big]\\
&=\mathbb{E}_k^\lambda\Bigg[\mathbf{1}_{\mathcal C}(X_k^\lambda)\,\exp\Big\{-k\,\sum_{i=1}^{n}\Big(\,\lambda_{i+1}(t_{i+1})X_k(t_{i+1})-\lambda_{i+1}(t_{i})X_k(t_{i})\\&\qquad\qquad\qquad\qquad\qquad-\int_{t_{i}}^{t_{i+1}} \!\!\![\,\lambda_{i+1}\,X_k(s)+ H(\lambda_{i+1}(s))\,]\,ds \,\Big)\Big\}\Bigg]\\
&\leq\sup_{\gamma\in \mathcal C}\, \exp\Bigg\{-k\,\sum_{i=1}^{n}\Big(\,\lambda_{i+1}(t_{i+1})\gamma(t_{i+1})-\lambda_{i+1}(t_{i})\gamma(t_{i})\\&\qquad\qquad\qquad\qquad\qquad-\int_{t_{i}}^{t_{i+1}} \!\!\![\,\lambda_{i+1}\,\gamma(s)+ H(\lambda_{i+1}(s))\,]\,ds \,\Big)\Bigg\}\\
&=\exp\Big\{-k\,\inf_{\gamma\in \mathcal C}\, \sum_{i=0}^{n-1}J^{i+1}_{\lambda_{i+1}}(\gamma)\Big\},\\
\end{split}
\end{equation*}
for all $\lambda_{i+1}:[t_i,t_{i+1}]\to\mathbb{R}$ linear function, where $J^{i+1}_{\lambda_{i+1}}(\gamma)$ is equal to
\begin{equation*}
\begin{split}
 & \lambda_{i+1}(t_{i+1})\gamma(t_{i+1})-\lambda_{i+1}(t_{i})\gamma(t_{i})-\int_{t_{i}}^{t_{i+1}} \!\!\![\,\lambda_{i+1}'(s)\,\gamma(s)+ H(\lambda_{i+1}(s))\,]\,ds .
\end{split}
\end{equation*}
Then, for all $\mathcal C$ open set on $D[0,T]$, minimizing over the time-partition and over  functions $\lambda_1,\dots,\lambda_n$, we have
\begin{equation*}
\begin{split}
&\varlimsup_{k\to\infty}\frac{1}{k}\log\mathbb{P}_k\Big[X_k\in \mathcal C\Big]
\leq-\sup_{\{t_i\}_i}\sup_{\lambda_1}\cdots\sup_{\lambda_n}\,\inf_{\gamma\in \mathcal C}\, \sum_{i=0}^{n-1}J^{i+1}_{\lambda_{i+1}}(\gamma).
\end{split}
\end{equation*}
Since $J^{i+1}_{\lambda_{i+1}}(\gamma)$ is continuous on $\gamma$,  using Lemma 3.3 (Minimax Lemma) in Appendix 2 of \cite{KL}, we can interchanged the supremum and infimum above. And, then, we obtain, for all $\mathcal K$ compact set
\begin{equation}\label{ineq1}
\begin{split}
&\varlimsup_{k\to\infty}\frac{1}{k}\log\mathbb{P}_k\Big[X_k\in \mathcal K\Big]
\leq-\inf_{\gamma\in \mathcal K} \,\sup_{\{t_i\}_i}\,I_{\{t_i\}}(\gamma),
\end{split}
\end{equation}
where $I_{\{t_i\}}(\gamma)=\sup_{\lambda_1}\cdots\sup_{\lambda_n}\,\sum_{i=0}^{n-1}J^{i+1}_{\lambda_{i+1}}(\gamma).$ Define $I(\gamma)=\sup_{\{t_i\}_i}\,I_{\{t_i\}}(\gamma)$. Notice that
\begin{equation*}
\begin{split}
\sup_{\lambda_1}\cdots\sup_{\lambda_n}\,\sum_{i=0}^{n-1}J^{i+1}_{\lambda_{i+1}}(\gamma)
=&\,\sup_{\lambda_1}J^{1}_{\lambda_{1}}(\gamma)+\cdots+\sup_{\lambda_n}J^{n}_{\lambda_{n}}(\gamma)\\
\geq&\,\sup_{\lambda\in \mathbb{R}}J^{1}_{\lambda}(\gamma)+\cdots+\sup_{\lambda \in \mathbb{R}}J^{n}_{\lambda}(\gamma)\,=\,\sum_{i=0}^{n-1} \sup_{\lambda\in \mathbb{R}}J^{i}_{\lambda}(\gamma).
\end{split}
\end{equation*}

If $\gamma\in\mathcal{AC}[0,T]$, then
\begin{equation*}
\begin{split}
J^{i}_{\lambda}(\gamma)
 &=\,(t_{i+1}-t_{i}) \,\Big\{\,\lambda\,\frac{1}{t_{i+1}-t_{i}}\int_{t_{i}}^{t_{i+1}} \!\!\!\!\gamma'(s)\,ds\,- \,\,H(\lambda) \Big\}.
\end{split}
\end{equation*}

Thus,
\begin{equation*}
\begin{split}
I_{\{t_i\}_i}(\gamma)\geq&\sum_{i=0}^{n-1}\,(t_{i+1}-t_{i}) \,\sup_{\lambda\in \mathbb{R}}\,\Big\{\lambda\,\frac{1}{t_{i+1}-t_{i}}\int_{t_{i}}^{t_{i+1}} \!\!\!\!\gamma'(s)\,ds\,- \,\,H(\lambda)\Big\}\\
&=\sum_{i=0}^{n-1}\,(t_{i+1}-t_{i}) \,L\Big(\frac{1}{t_{i+1}-t_{i}}\int_{t_{i}}^{t_{i+1}} \!\!\!\!\gamma'(s)\,ds\Big).\\
\end{split}
\end{equation*}
The last equality is true, because $L(v)=\sup_{\lambda\in\mathbb{R}}\{v\lambda-H(\lambda)\}$, see \eqref{func_melhor}.
Putting it on the definition of $I(\gamma)$, we have
\begin{equation}\label{(7)}
\begin{split}
I(\gamma)&=\sup_{\{t_i\}_i}\,I_{\{t_i\}_i}(\gamma)\\
&\geq  \sup_{\{t_i\}_i}\,\, \sum_{i=0}^{n-1}\,(t_{i+1}-t_{i}) \,L\Big(\frac{1}{t_{i+1}-t_{i}}\int_{t_{i}}^{t_{i+1}} \!\!\!\!\gamma'(s)\,ds\Big)\\
&\geq\int_0^TL(\gamma'(s))\,ds=I_T(\gamma),
\end{split}
\end{equation}
as on \eqref{funcional} or on \eqref{func_melhor}.

Now, consider the case where $\gamma\notin\mathcal{AC}[0,T]$, then
there is $\varepsilon>0$ such that for all $\delta>0$ there is a  family of intervals $\{(s_i,t_i)\}_{i=1}^{n}$ on $[0,T]$, with $\sum_{i=1}^{n} t_i-s_i<\delta$, but
$\sum_{i=1}^{n} \gamma(t_i)-\gamma(s_i)>\varepsilon $. Thus, taking the time-partition of $[0,T]$ as
$t'_0=0<t'_1<\dots<t'_{2n}<t'_{2n+1}=T$, over the points $s_i, t_i$, we get
\begin{equation*}
\begin{split}
 \sum_{j=1}^{2n} J_{\lambda}^j(\gamma)
 &=\lambda \sum_{j=1}^{2n} \gamma(t'_j)-\gamma(t'_{j-1})\,-\, H(\lambda)\sum_{j=1}^{2n} t'_j-t'_{j-1}\\
  &=\lambda \sum_{i=1}^{n} \gamma(t_i)-\gamma(s_i)\,-\,  H(\lambda)\sum_{i=1}^{n} t_i-s_i\\
&  \geq \lambda \varepsilon \,-\,H(\lambda)\delta.
\end{split}
\end{equation*}
Then, $$I(\gamma)\geq \lambda \varepsilon \,-\,H(\lambda)\delta,$$
for all $\delta>0$ and for all $\lambda\in \mathbb{R}$. Thus, $I(\gamma)\geq \lambda \varepsilon $, for all $\lambda\in \mathbb{R}$.  Remember that $\varepsilon $ is fixed and we take $\lambda\to \infty$. Therefore, $I(\gamma)=\infty$, for $\gamma\notin\mathcal{AC}[0,T]$.
Then, $I(\gamma)=I_{T}(\gamma)$ as on \eqref{funcional} or on \eqref{func_melhor}.

In conclusion, we have obtained,  by inequalities \eqref{ineq1}, \eqref{(7)}  and definition of $I(\gamma)$, that
\begin{equation*}
\begin{split}
&\varlimsup_{k\to\infty}\frac{1}{k}\log\mathbb{P}_k\Big[X_k\in \mathcal K\Big]
\leq-\inf_{\gamma\in \mathcal K} \,I_{T}(\gamma),
\end{split}
\end{equation*}
where $I_T$ was defined on \eqref{funcional} or on \eqref{func_melhor}.

\subsection{Upper bound for closed sets}\label{subsec2.3}
To extend the upper bound for closed sets we need to use a standard argument, which is to prove that the sequence of measures $\{\mathbb{P}_k\}_k$ is exponentially tight, see Proposition 4.3.2 on \cite{A} or on Section 1.2 of \cite{OV}.
By exponentially tight we understood that there is a sequence of compact sets $\{\mathcal{K}_j\}_j$ in $D[0,T]$ such that
\begin{equation*}
\begin{split}
&\varlimsup_{k\to\infty}\frac{1}{k}\log\mathbb{P}_k\Big[X_k\in \mathcal K_j\Big]
\leq-j,
\end{split}
\end{equation*}
for all $j\in \mathbb N$.

Then this section is concerned about exponential tightness. First of all, as in Section 4.3 on \cite{A} or in Section 10.4 on \cite{KL}, we also claim that the exponential tightness is just a consequence of the lemma below,

\begin{lemma}\label{l1}
 For every $\varepsilon >0$,
\begin{equation*}
 \varlimsup_{\delta\downarrow 0}\varlimsup_{k\to\infty}\frac{1}{k}\log \mathbb  P_{k}
\Big[\sup_{|t-s|\leq \delta}|X_k(t)-X_k(s)|>\varepsilon \Big]\,=\,\infty\,.
\end{equation*}
\end{lemma}

\begin{proof}
Firstly, notice that
\begin{equation*}
\begin{split}
&\Big\{\sup_{|t-s|\leq \delta}|\gamma(t)-\gamma(s)|>\varepsilon \Big\}\\
& \subset\bigcup_{k=0}^{\lfloor T\delta^{-1}\rfloor}\Big\{\sup_{k\delta\leq t< (k+1)\delta}
|\gamma(t)-\gamma(k\delta)|>\frac{\varepsilon }{4}\Big\}\,.\\
\end{split}
\end{equation*}
We have here $\pfrac{\varepsilon }{4}$ instead of $\pfrac{\varepsilon }{3}$ due to the presence of jumps.
Using the useful fact, for any sequence of real numbers $a_N,b_N$, we have
\begin{equation}\label{limsup}
 \varlimsup_{N\to\infty}\pfrac{1}{N}\log(a_N+b_N)=
 \max \Big\{\varlimsup_{N\to\infty}\pfrac{1}{N}\log(a_N),\varlimsup_{N\to\infty}\pfrac{1}{N}\log(b_N)\Big\}\,,
\end{equation}
in order to prove this lemma, it is enough to show that
\begin{equation}\label{l2}
 \varlimsup_{\delta\downarrow 0}\varlimsup_{k\to\infty}\pfrac{1}{k}\log \mathbb  P_{k}\Big[\sup_{t_0\leq t\leq t_0+\delta}
|X_k(t)-X_k(t_0)|>\varepsilon \Big]\,=\,\infty\,,
\end{equation}
for every $\varepsilon >0$ and for all $t_0\geq 0$. Let be $ M^{k}_t$ the martingale defined in \eqref{martingale1} with the function $\lambda$ constant, using the expression \eqref{martingale2} for $ M^{k}_t$ and the fact that $\lambda$ is constant,  we have that
\begin{equation*}
\begin{split}
 M^{k}_t\,=\, \exp{\Big\{k\big[c\lambda\,(X_k(t)-X_k(0))\,-\,t\,H(c\lambda)\big]\Big\}}
\end{split}
\end{equation*}
is a positive martingale equal to $1$ at time $0$. The constant $c$ above will be chosen \textit{a posteriori} as enough large.
In order to obtain \eqref{l2} is sufficient to get  the limits
\begin{equation}\label{l3}
 \varlimsup_{\delta\downarrow 0}\varlimsup_{k\to\infty}\pfrac{1}{k}\log \mathbb  P_{k}\Big[\sup_{t_0\leq t\leq t_0+\delta}\Big|
\pfrac{1}{k}\log \Big(\pfrac{M^{k}_t}{M^{k}_{t_0}}\Big) \Big|>c\lambda\,\varepsilon \Big]\,=\,-\infty\
\end{equation}
and
\begin{equation}\label{l4}
 \varlimsup_{\delta\downarrow 0}\varlimsup_{k\to\infty}\pfrac{1}{k}\log \mathbb  P_{k}\Big[\sup_{t_0\leq t\leq t_0+\delta}\Big|
(t-t_0)\, H(c\lambda) \Big|>c\lambda\varepsilon \Big]=-\infty\,.
\end{equation}
The second probability is considered for a  deterministic set, and
 by   boundedness,  we conclude that for $\delta$ enough small the probability
in \eqref{l4} vanishes.

On the other hand, to prove \eqref{l3}, we observe that we can neglect the absolute value, since
\begin{equation}\label{mod}
\begin{split}
&\mathbb  P_{k}\Big[\sup_{t_0\leq t\leq t_0+\delta}\Big|
\pfrac{1}{k}\log \Big(\pfrac{M^{k}_t}{M^{k}_{t_0}}\Big) \Big|>c\lambda\,\varepsilon \Big]\\
& \leq\mathbb  P_{k}\Big[\sup_{t_0\leq t\leq t_0+\delta}
\pfrac{1}{k}\log \Big(\pfrac{M^{k}_t}{M^{k}_{t_0}}\Big) >c\lambda\,\varepsilon \Big]+
 \mathbb  P_{k}\Big[\sup_{t_0\leq t\leq t_0+\delta}
\pfrac{1}{k}\log \Big(\pfrac{M^{k}_t}{M^{k}_{t_0}}\Big) <-c\lambda\,\varepsilon \Big]
\end{split}
\end{equation}
and using again \eqref{limsup}. Because $\{M^{k}_t/M^{k}_{t_0};\,t\geq t_0\}$ is a mean one positive martingale, we can apply Doob's Inequality,
which yields
\begin{equation*}
\mathbb  P_{k}\Big[\sup_{t_0\leq t\leq t_0+\delta}
\pfrac{1}{k}\log \Big(\pfrac{M^{k}_t}{M^{k}_{t_0}}\Big) >c\lambda\,\varepsilon \Big]
\,=\,
\mathbb  P_{k}\Big[\sup_{t_0\leq t\leq t_0+\delta}
 \Big(\pfrac{M^{k}_t}{M^{k}_{t_0}}\Big) >e^{c\lambda\,\varepsilon\, k }\Big]
 \,\leq\,\frac{1}{e^{c\lambda\varepsilon k}}\,.
\end{equation*}

 Passing the $\log$ function and dividing by $k$, we get
\begin{equation}\label{bound111}
 \varlimsup_{\delta\downarrow 0}\varlimsup_{k\to\infty}\pfrac{1}{k}\log
\mathbb  P_{k}\Big[\sup_{t_0\leq t\leq t_0+\delta}
\pfrac{1}{k}\log \Big(\pfrac{M^{k}_t}{M^{k}_{t_0}}\Big) >\lambda\,\varepsilon \Big]\leq -c\lambda\,\varepsilon ,
\end{equation}
for all $c>0$.
To treat of the second term on \eqref{mod}, we just need to observe that $\{M^{k}_{t_0}/M^{k}_{t};\,t\geq t_0\}$ is also a martingale and rewriting
$$\mathbb  P_{k}\Big[\sup_{t_0\leq t\leq t_0+\delta}
\pfrac{1}{k}\log \Big(\pfrac{M^{k}_t}{M^{k}_{t_0}}\Big) <-c\lambda\,\varepsilon \Big]$$
as
$$\mathbb  P_{k}\Big[\sup_{t_0\leq t\leq t_0+\delta}
\pfrac{1}{k}\log \Big(\pfrac{M^{k}_{t_0}}{M^{k}_{t}}\Big) >c\lambda\,\varepsilon \Big].$$
Then, we get the same bound for this probability  as in \eqref{bound111}, it finishes the proof.
\end{proof}

\subsection{Lower bound}\label{subsec2.4}
Let $\gamma:[0,T]\to\mathbb{S}^{1}$ be a function such that $\gamma(0)=x_0$ and for a $\delta>0$, in the following $$B_\infty(\gamma,\delta)=\Big\{f:[0,T]\to\mathbb{S}^{1}:\, \sup_{0\leq t\leq T}|f(t)-\gamma(t)|<\delta\Big\}.$$

Let $\mathcal O$ be a open set of $D[0,T]$. For all $\gamma \in\mathcal O$,
our goal is prove that
\begin{equation}\label{**}
\varliminf_{k\to\infty}\frac{1}{k}\log\mathbb P_k[X_k\in\mathcal O]\geq -I_T(\gamma).
\end{equation}
For that, we can suppose $\gamma\in\mathcal{AC}[0,T]$, because if $\gamma\notin\mathcal{AC}[0,T]$, then $I_T(\gamma)=infty$ and \eqref{**} is trivial.
Since $\gamma \in\mathcal O$, there is a $\delta>0$ such that
$$\mathbb{P}_k\Big[X_k\in \mathcal O\Big]\geq \mathbb{P}_k\Big[X_k\in B_\infty(\gamma,\delta)\Big].$$
We need consider the measure $\mathbb{P}_k^\lambda$ with $\lambda:[0,T]\to\mathbb R$, the function  $\lambda(s)=\lambda_\gamma(s)=\log\Big(\pfrac{1}{2}\Big(\gamma'(s)+\sqrt{(\gamma'(s))^2+4}\Big)\Big)$, which we obtain in the Lemma \ref{Legendre}, as a function that attains the supremum $\sup_\lambda[ \lambda\, \gamma'(s)-H(\lambda)] $ for each $s$.
 Thus,
\begin{equation*}
\begin{split}
&\mathbb{P}_k\Big[X_k\in B_\infty(\gamma,\delta)\Big]
=\mathbb{E}_k^{\lambda}
\Big[\mathbf{1}_{B_\infty(\gamma,\delta)}(X_k^ {\lambda})\frac{d\mathbb{P}_k}{d\mathbb{P}_k^{\lambda}}\Big]=\mathbb{E}_k^{\lambda}\Big[
\mathbf{1}_{B_\infty(\gamma,\delta)}(X_k^{\lambda})(M^k_T)^{-1}\Big]\\
&=\mathbb{E}_k^{\lambda}
\Bigg[\mathbf{1}_{B_\infty(\gamma,\delta)}(X_k^{\lambda})\,\,\exp\Big\{k\,\Big[\lambda(T)X_k(T)-\lambda(0)X_k(0)\\&\qquad\qquad\qquad\qquad\qquad\qquad
-\int_{0}^{T} \![\,\lambda'(s)\,X_k(s)+ H(\lambda(s))\,]\,ds \,\Big]\Big\}\Bigg].\\
\end{split}
\end{equation*}
The last equality follows from Remark \ref{lambda_dif}.
Define the measure $\mathbb{P}_{k,\delta}^{\lambda,\gamma}$ as
\begin{equation}\label{measure}
\begin{split}
&\mathbb{E}_{k,\delta}^{\lambda,\gamma}
\Big[f(X_k^{\lambda})\Big]=
\frac{\mathbb{E}_k^{\lambda}
\Big[\mathbf{1}_{B_\infty(\gamma,\delta)}(X_k^{\lambda})f(X_k^{\lambda})\Big]}{\mathbb{P}_k^{\lambda}[X_k^{\lambda}\in B_\infty(\gamma,\delta)]},
\end{split}
\end{equation}
for all bounded function $f:D[0,T]\to\mathbb R$.
Then,
\begin{equation*}
\begin{split}
&\mathbb{P}_k\Big[X_k\in B_\infty(\gamma,\delta)\Big]\\
&=\mathbb{E}_{k,\delta}^{\lambda,\gamma}
\Big[\exp\Big\{-k\,\big[\,\lambda(T)\,X_k^{\lambda}(T)-\lambda(0)\,X_k^{\lambda}(0))\,-\int_{0}^{T} \!\lambda'(s)\,X_k(s)\,ds \big]\Big\}\Big] \\
&\qquad\qquad\cdot\,e^{ k\int_{0}^{T} \! H(\lambda(s))\,ds }\,\,
\mathbb{P}_k^{\lambda}\Big[X_k^{\lambda}\in B_\infty(\gamma,\delta)\Big].\\
\end{split}
\end{equation*}
Then, using Jensen's inequality
\begin{equation*}
\begin{split}
&\frac{1}{k}\log\mathbb{P}_k\Big[X_k\in \mathcal O\Big]\\
&\geq - \,\mathbb{E}_{k,\delta}^{\lambda,\gamma}
\Big[\lambda(T)\,X_k^{\lambda}(T)-\lambda(0)\,X_k^{\lambda}(0))\,-\int_{0}^{T} \!\lambda'(s)\,X_k(s)\,ds \Big]\\&\qquad\qquad\qquad+\,\int_{0}^{T} \! H(\lambda(s))\,ds+\frac{1}{k}\log\mathbb{P}_k^{\lambda}\Big[X_k^{\lambda}\in B_\infty(\gamma,\delta)\Big]\\
&\geq- \,C(\lambda)\mathbb{E}_{k,\delta}^{\lambda,\gamma}
\Big[|X_k^{\lambda}(T)-\gamma(T)|+|X_k^{\lambda}(0))-\gamma(0)|+\!\int_{0}^{T} \!\!|X_k(s)-\gamma(s)|\,ds \Big]\\&
\qquad\qquad\qquad-\Big(\lambda(T)\,\gamma(T)-\lambda(0)\,\gamma(0))\,-\int_{0}^{T} \![\lambda'(s)\,\gamma(s)+H(\lambda(s))]\,ds\Big)\\&\qquad\qquad\qquad+\frac{1}{k}\log\mathbb{P}_k^{\lambda}\Big[X_k^{\lambda}\in B_\infty(\gamma,\delta)\Big].\\
\end{split}
\end{equation*}
Since $\gamma:[0,T]\to\mathbb{R}$ is  an absolutely continuous function, we can write
\begin{equation*}
\begin{split}
&\lambda(T)\,\gamma(T)-\lambda(0)\,\gamma(0))\,-\int_{0}^{T} \![\lambda'(s)\,\gamma(s)+H(\lambda(s))]\,ds\\
&=
\int_{0}^{T} \![\lambda(s)\,\gamma'(s)+H(\lambda(s))]\,ds.
\end{split}
\end{equation*}
Since $\lambda(s)=\lambda_\gamma(s)=\log\Big(\pfrac{1}{2}\Big(\gamma'(s)+\sqrt{(\gamma'(s))^2+4}\Big)\Big)$, by Lemma \ref{Legendre},  we obtain
\begin{equation*}
\begin{split}
&\int_{0}^{T} \![\lambda(s)\,\gamma'(s)+H(\lambda(s))]\,ds=\int_0^T\sup_\lambda[ \lambda\, \gamma'(s)-H(\lambda)]\,ds=\int_0^T L(\gamma'(s))\,ds,
\end{split}
\end{equation*}
and, by \eqref{func_melhor}, the last expression is equal to $I_{T}(\gamma)$.
Thus,
\begin{equation}\label{eqq00}
\begin{split}
&\frac{1}{k}\log\mathbb{P}_k\Big[X_k\in \mathcal O\Big]\geq -I_{T}(\,\gamma)+\frac{1}{k}\log\frac{3}{4} -C(\lambda)\delta.
\end{split}
\end{equation}
The last inequality follows from the above and the Lemma \ref{mart2} and the Lemma \ref{bola} below.

\begin{lemma}\label{mart2} With respect the measure defined on \eqref{measure}, there exists a constant $C>0$ such that
\begin{equation*}
-\mathbb{E}_{k,\delta}^{\lambda,\gamma}
\Big[|X_k^{\lambda}(T)-\gamma(T)|+|X_k^{\lambda}(0))-\gamma(0)|+\!\int_{0}^{T} \!\!|X_k(s)-\gamma(s)|\,ds \Big]\geq -C\delta.
\end{equation*}
\end{lemma}

\begin{lemma}\label{bola} There is a $k_0=k_0(\gamma,\delta)$ such that
$\mathbb{P}_k^{\lambda}[X_k^{\lambda}\in B_\infty(\gamma,\delta)]>\frac{3}{4}$,
 for all $k\geq k_0$.
\end{lemma}
The proofs of Lemma \ref{mart2} and Lemma \ref{bola} are in the end of this subsection.\\

Continuing with the analysis of \eqref{eqq00}, we mention that,
since, for all $\gamma\in\mathcal O$, there exists $\delta=\delta(\gamma)$, such that $B_\infty(\gamma,\delta)\subset\mathcal O$, then for all $\varepsilon<\delta$, we have
\begin{equation*}
\begin{split}
&\varliminf_{k\to\infty}\frac{1}{k}\log\mathbb{P}_k\Big[X_k\in \mathcal O\Big]\geq -I_{T}(\gamma) -\lambda\varepsilon.
\end{split}
\end{equation*}
Thus, for all $\gamma\in\mathcal O$, we have
\eqref{**}.
Therefore,
\begin{equation*}
\begin{split}
&\varliminf_{k\to\infty}\frac{1}{k}\log\mathbb{P}_k\Big[X_k\in \mathcal O\Big]\geq -\inf_{\gamma\in \mathcal O}I_{T}(\gamma) .
\end{split}
\end{equation*}

We present, now, the proofs of the Lemmata \ref{mart2} and \ref{bola}.

\begin{proof}[Proof of Lemma \ref{mart2}]
Recalling the definition of the probability measure $\mathbb{P}_{k,\delta}^{\lambda,\gamma}$, we can write
\begin{equation*}
\begin{split}
&-\mathbb{E}_{k,\delta}^{\lambda,\gamma}
\Big[|X_k^{\lambda}(T)-\gamma(T)|+|X_k^{\lambda}(0))-\gamma(0)|+\!\int_{0}^{T} \!\!|X_k(s)-\gamma(s)|\,ds \Big]\\&
=-\frac{\mathbb{E}_k^{\lambda}
\Big[\mathbf{1}_{B_\infty(\gamma,\delta)}\Big(|X_k^{\lambda}(T)-\gamma(T)|+|X_k^{\lambda}(0))-\gamma(0)|+\!\int_{0}^{T} \!\!|X_k(s)-\gamma(s)|\,ds \Big)\Big]}{\mathbb{P}_k^{\lambda}[X_k^{\lambda}\in B_\infty(\gamma,\delta)]}\\
&\geq-(2+T)\,\delta\,\frac{\mathbb{P}_k^{\lambda}
\Big[X_k^{\lambda}\in B_\infty(\gamma,\delta)\Big]}{\mathbb{P}_k^{\lambda}[X_k^{\lambda}\in B_\infty(\gamma,\delta)]}=\,-\,(2+T)\,\delta.\\
\end{split}
\end{equation*}
\end{proof}

\begin{proof}[Proof of Lemma \ref{bola}]
Consider the martingale
\begin{equation*}
\begin{split}
\mathcal{M}^k_t&=X_k^{\lambda}(t)-X_k^{\lambda}(0)-\int_0^t\!\! k\mathcal{L}_k^{\lambda} X_k^{\lambda}(s)\, ds\\
&=X_k^{\lambda}(t)-\pfrac{\lfloor kx_0\rfloor}{k}-\int_0^t\!\!\!\big(e^{\lambda(s)}-e^{-\lambda(s)}\big)\,ds,\\
\end{split}
\end{equation*}
remember that $\mathbb{P}_k$ has initial measure $\delta_{x_k(x_0)}$, where $x_k(x_0)=\frac{\lfloor kx_0\rfloor}{k}$.
Notice that, by the choose of $ \lambda(s)$ as $\log\Big(\pfrac{1}{2}\Big(\gamma'(s)+\sqrt{(\gamma'(s))^2+4}\Big)\Big)$ and hypothesis over $\gamma$, we have that
\begin{equation*}
\begin{split}
&\int_0^t\!\!\!\big(e^{\lambda(s)}-e^{-\lambda(s)}\big)\,ds=\int_0^t
\gamma'(s)\,ds=\gamma(t)-\gamma(0)=\gamma(t)-x_0.\\
\end{split}
\end{equation*}
Then,
$X_k^{\lambda}(t)-\gamma(t)=\mathcal{M}^k_t+r_k$,
 where $r_k=\frac{\lfloor kx_0\rfloor}{k}-x_0$.
Using  the  Doob's martingale inequality,
\begin{equation}\label{doob}
\begin{split}
\mathbb{P}_k^{\lambda}\Bigg[\sup_{0\leq t\leq T}|X_k^{\lambda}(t)-\gamma(t)|>\delta \Bigg]&\leq
\mathbb{P}_k^{\lambda}\Bigg[\sup_{0\leq t\leq T}|\mathcal{M}^k_t|>\delta/2\Bigg]+\mathbb{P}_k^{\lambda}\Bigg[|r_k|>\delta/2 \Bigg]
\\&\leq\frac{4}{\delta^2}\, \mathbb{E}_k^{\lambda}\Big[\big(\mathcal{M}^k_T\big)^2\Big]+\frac{1}{8},
\end{split}
\end{equation}
for $k$ large enough.
Using the fact that
\begin{equation*}
\begin{split}\mathbb{E}_k^{\lambda}\Big[\big(\mathcal{M}^k_T\big)^2\Big]
=&\mathbb{E}_k^{\lambda}\Big[\int_0^T[\,k\mathcal{L}_k^{\lambda} (X_k^{\lambda}(s))^2-2X_k^{\lambda}(s)k\mathcal{L}_k^{\lambda} (X_k^{\lambda}(s))\,]\, ds\Big].
\end{split}
\end{equation*}
And,
making same more calculations, we get that the expectation above is bounded from above by
\begin{equation*}
\begin{split}
&\mathbb{E}_k^{\lambda}\Bigg[k\int_0^Te^{\lambda(s)}\big((X_k^{\lambda}(s)+\pfrac{1}{k})-X_k^{\lambda}(s))\big)^2\,ds\Bigg]\\
&+\mathbb{E}_k^{\lambda}\Bigg[k\int_0^T
e^{-\lambda(s)}\big( (X_k^{\lambda}(s)-\pfrac{1}{k})-(X_k^{\lambda}(s))\big)^2\, ds\Bigg]\\
&=\int_0^T\frac{e^{\lambda(s)}+e^{-\lambda(s)}}{k}\,ds\leq C(\lambda,T)\frac{1}{k}.
\end{split}
\end{equation*}
Then there is $k_0$, such that,
$\mathbb{P}_k^{\lambda}[\sup_{0\leq t\leq T}|X_k^{\lambda}(t)-\gamma(t)|>\delta ]<1/4$, for all $k>k_0$.

\end{proof}
\medskip

\textit{This is the end of the first part of the paper where we investigate the deviation function on the Skorohod space when $k\to \infty$ for the trajectories of the unperturbed system.}

\section{Disturbing the system by a potential $V$.}\label{sec3}

Now, we introduce a fixed differentiable $C^2$  function $V:  \mathbb{S}^1 \to \mathbb{R}.$ We want to analyse large deviation properties associated to the  disturbed system by the potential $V$. Several of the properties we consider just assume that $V$ is Lipschitz, but we need some more regularity for Aubry-Mather theory.
Given $V: \mathbb{S}^1 \to \mathbb{R}$   we denote by $V_k$ the restriction of
$V$ to $\Gamma_k$.
It is known that if $kL_k$ is a $k$ by $k$ line sum zero matrix with strictly negative elements in the diagonal and non-negative elements outside the diagonal, then for any $t>0$, we have that $e^{t\,kL_k}$ is stochastic. The infinitesimal generator $kL_k$ generates a continuous time Markov Chain with values on $\Gamma_k=\{0,1/k, 2/k,...,\frac{k-1}{k}\}\subset \mathbb S^1$. We are going to disturb this stochastic semigroup by a potential $k\,V_k:\Gamma_k\to \mathbb{R}$ and we will derive another continuous Markov Chain (see \cite{BEL}  and \cite{LNT}) with values on $\Gamma_k$. This will be described below. We will identify the function $k\,V_k$ with the $k$ by $k$  diagonal matrix, also denoted by $k\,V_k$, with elements $k\,V_k(j/k)$, $j=0,1,2..,k-1$, in the diagonal.

\medskip

The continuous time Perron's Theorem (see \cite{S}, page 111) claims the following: given the matrix $ k\,L_k$ as above and the $k\,V_k$ diagonal matrix, then there exists
\begin{itemize}
\item[a)] a unique positive function $u_{V_k}=u_k : \{0,1/k,2/k,..,(k-1)/k\}\to \mathbb{R}$,
\item[b)] a unique probability vector  $\mu_{V_k}=\mu_k$  over the set $  \{0,1/k,2/k,..,(k-1)/k\}$, such that
$$ \sum_{j=1}^k
u_k^j \,\mu_k^j = 1 ,$$
where $u_k=(u_k^1,...,u_k^k)$, $\mu_k=(\mu_k^1,...,\mu_k^k)$
\item[c)]a real value $\lambda (V_k)=\lambda_k$,

\end{itemize}

such that

\begin{itemize}
\item[i)] for any  $v \in \mathbb{R}^n$, if we denote $ P^t_{k,V} =e^{t\,(k\,L_k + k\,V_k)}$, then
$$\lim_{t\to \infty}  e^{-t \lambda (k)} P^t_{k,V} (v)  = \,\sum_{j=1}^k
v_j \,\mu_k^j\, u_k^j\,, $$
\item[ii)] for any positive $s$
$$ e^{-s \lambda (k)}P^s_{k,V}(u_k)= u_k. $$
\end{itemize}

From ii) follows that
$$(k\,L_k + k\,V_k) (u_k) = \lambda (k) u_k.$$

The semigroup $e^{t\, (k\, L_k + k\,V_k - \lambda(k))}$  defines a continuous time Markov chain with values on $ \Gamma_k$, where the vector $\pi_{k,V}=(\pi_{k,V}^1,...,\pi_{k,V}^k)$, such that $\pi_{k,V}^j=\,u_k^j\, \mu_k^j\, \,$, $j=1,2,..,k$, is stationary.
Notice that $\pi_k=\pi_{k,V}$, when $V=0$.
Remember that the $V_k$ was obtained by discretization of the initial $V:\mathbb S^1\to \mathbb{R}.$

\medskip

\begin{example} When $k=4$ and $V_4$ is defined by the values $V_4^j$, $j=1,2,3,4$, then, we have first to find the left eigenvector
$u_{V_4}$ for the eigenvalue $\lambda(V_4)$, that is to solve the equation

$$ u_{V_4}\, (4L_4 +4V_4)= u_{V_4} 4
\left(
\begin{array}{cccc}
-2  + V_4^1 & 1 & 0 & 1 \\
1 & -2  + V_4^2 & 1 & 0 \\
0 & 1 & -2+ V_4^3 & 1 \\
1 & 0 & 1 & -2 + V_4^4\\
\end{array}\right)=
 \lambda(V_4)\, u_{V_4}.
$$

Suppose $\mu_{V_4}$ is the right normalized eigenvector. In this  way we can get by the last theorem a stationary vector $\pi_{4,V}$ for
stationary Gibbs probability associated to the potential $V_4$ We point out that by numeric methods one can get good approximations of the
solution of the above problem.

\end{example}

\bigskip

From the end of Section 5 in \cite{S}, we have that

$$ \lambda_k = \sup_{ \psi \in \mathbb{ L}^2, \, ||\psi||_2=1}\Big\{
\int_{\Gamma_k} \psi (x)\,  [(k L_k + k V_k) ( \psi )\,
] (x) \, d \pi_{k} (x)\Big\},$$ where $\psi:\Gamma_k \to
\mathbb{R}$,  $$||\psi||_2= \sqrt{\frac{1}{k}\sum_{j=0}^{k-1}
\psi(\pfrac{j}{k})^2},$$ and $\pi_{k}$ is uniform in
$\Gamma_k$.
Notice that for any $\psi$, we have
$$\int_{\Gamma_k} \psi (x)\,  (k L_k ) ( \psi )
 (x) \, d \pi_k (x)=-\sum_{j=0}^{k-1}(\psi(\pfrac{j+1}{k})-\psi(\pfrac{j}{k}))^2.$$
Moreover,
$$ \int_{\Gamma_k} \psi (x)\,  [(k L_k + k\, V_k) ( \psi )\,
] (x) \, d \pi_k (x)=\sum_{j=0}^{k-1}[-(\psi(\pfrac{j+1}{k})-\psi(\pfrac{j}{k}))^2+\psi(\pfrac{j}{k})^2V_(\pfrac{j}{k}))].$$
In this way
$$ \pfrac{1}{k}\lambda_k = \sup_{ \psi \in \mathbb{ L}^2, \, ||\psi||_2=1}\Big\{
 \frac{1}{k}\int_{\Gamma_k} \psi (x)\,  [(k L_k + k V_k) ( \psi )\,
] (x) \, d \pi_k (x)\Big\}$$
$$= \sup_{ \psi \in \mathbb{ L}^2, \, ||\psi||_2=1} \Big\{-\frac{1}{k} \sum_{j=0}^{k-1}(\psi(\pfrac{j+1}{k})-\psi(\pfrac{j}{k}))^2+\frac{1}{k} \sum_{j=0}^{k-1}\psi(\pfrac{j}{k})^2V_k(\pfrac{j}{k})\Big\}.$$
Observe that for any $\psi\in \mathbb{ L}^2$, with $||\psi||_2=1$, the expression inside the braces is bounded from above by
$$ \frac{1}{k} \sum_{j=0}^{k-1}\psi(\pfrac{j}{k})^2V_k(\pfrac{j}{k})\leq\sup_{x\in\mathbb{S}^1}V(x).$$
Notice that for each $k$ fixed, the vector
$\psi^k=\psi$ that attains the maximal value $\lambda_k$ is
such that $\psi_k^i= \sqrt{u_{k,V}^i}$, with $i\in
\{0,...,(k-1)\}$,
$$ \sup_{ \psi \in \mathbb{ L}^2, \, ||\psi||_2=1}\Big\{
 \frac{1}{k}\int_{\Gamma_k} \psi (x)\,  [(k L_k + k V_k) ( \psi )\,
] (x) \, d \pi_k (x)\Big\}$$$$=-\int_{\Gamma_k} \psi_k (x)\,  [(k L_k + k\, V_k) ( \psi_k )\,
] (x) \, d \pi_k (x) =\pfrac{1}{k}\lambda_k.$$
When $k$ is large the above $\psi_k$ have the tendency to become more and more sharp close to the maximimum of $V_k$. Then, we have
that
$$\sup_{ \psi \in \mathbb{ L}^2, \, ||\psi||_2=1}\Big\{
 \pfrac{1}{k}\int_{\Gamma_k} \psi (x)\,  [(k L_k + k V_k) ( \psi )\,
] (x) \, d \pi_k (x)\Big\}$$
converges to
$$\sup_{ \psi \in \mathbb{ L}^2(dx), \, ||\psi||_2=1}\Big\{ \int_{\mathbb{S}^1}\, \psi
(x)\,  V(x) \, \psi  (x) \, d x\, \Big\}=\sup \{V(x)\,|\, x \in
 \mathbb{S}^1\, \} ,$$
when $k$ increases to $\infty$.

\medskip

Summarizing, we get the proposition below:

\begin{proposition}
$$ \lim_{k\to\infty}\pfrac{1}{k}\, \lambda_k =
\sup_{ \psi \in \mathbb{ L}^2(d x), \, ||\psi||_2=1}\Big\{ \int_{\mathbb{S}^1}\, \psi
(x)\,  V(x) \, \psi  (x) \, d x\, \Big\}$$
$$=\sup \{V(x)\,|\, x \in
 \mathbb{S}^1\, \} = - \inf_{\mu} \Big\{\int\!\! L(x,v)\, d \mu (x,v)\Big\},$$
where the last infimum is taken over all measures $\mu$ such that  $\mu$ is invariant probability for the Euler-Lagrange flow of $ L( x,v)$.
\end{proposition}

The last equality follows from Aubry-Mather theory (see \cite{CI} and \cite{Fath}).
Notice that this Lagrangian is convex and superlinear.

\medskip

\subsection{Lax-Oleinik semigroup}\label{subsec3.1}

By Feynman-Kac, see \cite{KL}, we have that the semigroup associated to the infinitesimal generator $k\,\mathcal L_k+kV_k$ has the following expression
$$P^t_{k,V}(f)(x) =\mathbb{E}_k\big[e^{\int_0^t kV_k(X_k(s))\,ds}f(X_k(t))\big],$$
for all bounded mensurable function $f:\mathbb S^1\to \mathbb R$ and all $t\geq 0$.

Now, consider
$$P^{T}_{k,V}(e^{ku})(x) =\mathbb{E}_k\big[e^{k\,[\int_0^{T} \!V_k(X_k(s))\,ds\,+\,u(X_k(T))\,]}\big],$$
for a fixed Lipschitz function $u:\mathbb S^1\to \mathbb R$.
Now, we want to use the results of Section \ref{sec2} together with the Varadhan's Lemma, which is

\begin{lemma}[Varadhan's Lemma (see \cite{DZ})]  Let $\mathcal E$ be a regular topological space; let $(Z_t)_{t>0}$ be a family of random variables taking values in $\mathcal E$; let $\mu_\varepsilon$ be the law (probability measure) of $Z_t$. Suppose that $\{\mu_\varepsilon\}_{\varepsilon>0}$ satisfies the large deviation principle with good rate function $I : \mathcal E\to [0, +\infty]$. Let $\phi  : \mathcal E \to \mathbb R$ be any continuous function. Suppose that at least one of the following two conditions holds true: either the tail condition
$$\lim_{M \to \infty} \varlimsup_{\varepsilon \to 0} \varepsilon \log \mathbb{E} \big[ \exp \big( \phi(Z_{\varepsilon}) / \varepsilon \big) \mathbf{1} \big( \phi(Z_{\varepsilon}) \geq M \big) \big] =  - \infty,$$
where $\mathbf 1(A)$ denotes the indicator function of the event $A$; or, for some $\gamma > 1$, the moment condition
$$\varlimsup_{\varepsilon \to 0} \varepsilon \log \mathbb{E} \big[ \exp \big( \gamma \phi(Z_{\varepsilon}) / \varepsilon \big) \big] < + \infty.$$
Then,
$$\lim_{\varepsilon \to 0} \varepsilon \log \mathbb{E} \big[ \exp \big( \phi(Z_{\varepsilon}) /\varepsilon \big) \big] = \sup_{x \in \mathcal E} \big( \phi(x) - I(x) \big).$$
\end{lemma}

We will consider here the above $\varepsilon$ as $\frac{1}{k}.$
By Theorem \ref{teo1} and Varadhan's Lemma, for each  Lipschitz function $u:\mathbb S^1\to \mathbb R$, we have
\begin{equation}
\begin{split}
\lim_{k\to\infty}\pfrac{1}{k}\log \,P^{T}_{k,V}(e^{ku})(x) &=
\lim_{k\to\infty}\pfrac{1}{k}\log\mathbb{E}_k\big[e^{k\,[\int_0^{T} \!V_k(X_k(s))\,ds\,+\,u(X_k(T))\,]}\big]\\
&=\sup_{\gamma\in D[0,T]}
\Big\{\int_0^{T} V(\gamma(s))\,ds+u(\gamma(T))-I_T(\gamma)\Big\}
\end{split}
\end{equation}
When $\gamma \notin AC[0,T]$, $I_T(\gamma)=\infty$ and if $\gamma \in AC[0,T]$, $I_T(\gamma)=\int_0^TL(\gamma'(s))\,ds$. Thus,
$$\lim_{k\to\infty}\pfrac{1}{k}\log \,P^{T}_{k,V}(e^{ku})(x) =\sup_{\gamma\in AC[0,T]}
\Big\{\,u(\gamma(T))\,-\,\int_0^T\!\!\big[L(\gamma'(s))-V(\gamma(s))\big]\,ds\Big\}.$$

For a fixed $T>0$, define the operator $\mathcal{T}_T $ acting on Lipschitz functions $u:\mathbb S^1\to \mathbb R$ by the expression
$\mathcal{T}_T(u)(x)=\lim_{k\to\infty}\pfrac{1}{k}\log \,P^{T}_{k,V}(e^{ku})(x)$, then, we just show that
$$\mathcal{T}_T(u)(x)\,\,=\sup_{\gamma\in AC[0,T]}
\Big\{\,u(\gamma(T))\,-\,\int_0^T\!\!\big[L(\gamma'(s))-V(\gamma(s))\big]\,ds\Big\}.$$
This family of operators parametrized by $T>0$ and acting on function $u:\mathbb S^1 \to \mathbb{R}$ is called the Lax-Oleinik semigroup.

\subsection{The Aubry-Mather theory}
We will use now Aubry-Mather theory (see \cite{CI} and \cite{Fath}) to obtain a fixed point $u$ for such operator. This will be necessary later in next section. We will elaborate on that.
Consider Mather measures, see \cite{Fath} and  \cite{CI}, on the circle $\mathbb{S}^1$ for the Lagrangian
\begin{equation}\label{L}
L^V(x,v)= - V(x) + v \log(  (v + \sqrt{v^2 + 4} )/2 ) - \sqrt{v^2
+ 4} + 2,
\end{equation}
 $x\in \mathbb S^1, v \in T_x \mathbb S^1$, when $V: \mathbb S^1\to \mathbb{R}$ is a $C^2$ function. This will be Delta Dirac on any of the points of  $\mathbb S^1$, where $V$ has maximum (or convex combinations of them). In order to avoid technical problems  we will assume that this point $x_0$ where the maximum is attained is unique. This is generic among $C^2$ potentials $V$.

This Lagrangian appeared in a natural way, when we analysed the
asymptotic deviation depending on  $k\to \infty$
for the discrete state space continuous time Markov Chains indexed by $k$, $\{X_k(t),t\geq 0\}$,
described above in Section \ref{sec2}.
We denote by $H(x,p)$ the associated Hamiltonian obtained via Legendre transform.

Suppose $u_+$ is a fixed point for the positive Lax-Oleinik semigroup and $u_{-}$ is a fixed point for the
negative Lax-Oleinik semigroup (see next  section for precise definitions). We will show that function $I^V= u_+ + u_{-}$ defined on $ \mathbb{S}^1$ is the
deviation function for  $\pi_{k,V} $, when $k\to\infty.$

\textit{Fixed functions $u$ for the Lax-Oleinik operator are weak KAM solutions of the Hamilton-Jacobi equation for the  corresponding Hamiltonian $H$ (see Sections 4 and 7 in \cite{Fat}).}

The so called
critical value in Aubry-Mather theory is
$$c(L)=- \inf_{\mu} \, \int L^V(x,v) d \mu(x,v)=\sup\{V(x)\,|\,x\in  \mathbb{S}^1\},$$
where the infimum above is taken over all measures $\mu$ such that $\mu$ is invariant probability for
the Euler-Lagrange flow $L^V$.
Notice that
\begin{equation}\label{*}
\lim_{k\to\infty}\frac{1}{k}\, \lambda_k =c(L).
\end{equation}
This will play an important role in what follows. A Mather measure is any $\mu$ which attains the above infimum value. This minimizing probability is defined on the tangent bundle of $\mathbb{S}^1$ but as it is a graph (see \cite{CI}) it can be seen as a probability on $\mathbb{S}^1$. This will be our point of view.

In the case that the potential $V$ has a unique point $x_0$  of maximum on $ \mathbb{S}^1$,
we have that $c(L)=V(x_0)$. The Mather measure in this case is a Delta Dirac on the point $x_0$.

Suppose there exist two points $x_1$ and $x_2$ in $ \mathbb{S}^1$, where
the supremum of the potential $V$ is attained. For the above defined lagrangian
$L$ the static points are $(x_1,0)$ and $(x_2,0)$ (see \cite{CI} and \cite{Fat} for definitions and general references on Mather Theory).
This case requires a more complex analysis, because it requires some hypothesis in order to know which of the points $x_0$ or $x_1$  the larger part of the mass of $\pi_{k,V}$ will select. We will not analyse such problem here.
In this case the
critical value is
$c(L)=-\, L^V(x_1,0)= V(x_1)= -\, L^V(x_2,0)=V(x_2).$

In appendix of \cite{A1}  and also in \cite{A2} the N. Anantharaman shows, for $t$ fixed,  an
interesting result relating the time  re-scaling of the Brownian
motion $B(\varepsilon t)$, $k\to\infty,$ and Large Deviations. The large deviation is obtained via Aubry-Mather theory. The convex part of the Mechanical Lagrangian in this case is $\frac{1}{2}\, |v|^2$. When there are two points $x_1$ and $x_2$ of maximum for $V$ the same problem as we mention before  happens in this other setting: when $\varepsilon\to 0$, which is the selected Mather measure? In this setting partial answers to this problem is obtained  in \cite{AIP}.

In the present paper we want to obtain similar results for $t$ fixed, but
for the re-scaled  semigroup $P_{k}(ks)=e^{skL_k}$, $s \geq 0 $, obtained
from the speed up by $k$ the time of the continuous time symmetric random walk (with the compactness
assumption) as described above.

In other words we are considering that the unitary circle (the
interval $[0,1)$) is being approximated by a discretization by $k$
equally spaced points, namely, $\Gamma_k=
\{0,1/k,2/k,...,(k-1)/k\}$.

Let $ \mathbb{ X}_{t,x}$ be the set of absolutely  continuous
paths $\gamma:[0,t)\to [0,1]$, such that $\gamma(0)=x$.

Consider the positive Lax-Oleinik operator acting on continuous function
$u$ on the circle: for all $t>0$
$$(\mathcal{T}^+_t (u))\, (x)=$$
$$ \sup_{\gamma \in \mathbb{ X}_{t,x}} \!\Big\{  u(\gamma(t)) - \int_0^t \!\!\big[(
\dot{\gamma}(s) \log \Big(\frac{ ( \dot{\gamma}(s) +
\sqrt{\dot{\gamma}^2 (s) + 4} }{2} \Big) - \sqrt{\dot{\gamma}^2(s) +
4} + 2 - V (\gamma(s))\big] \,d s \Big\}.
$$
It is well known (see \cite{CI} and \cite{Fath})  that there exists a Lipschitz function $u_+$ and
a constant $c=c(L)$ such that for all $t>0$
$$\mathcal{T}^+_t (u_+) = u_+  + c \, t. $$
We say that   $u_+$ is a $(+)$-solution of the Lax-Oleinik
equation.
This function $u_+$ is not always unique. If we add a constant to $u_+$ get another fixed point.
To say that the fixed point $u_+$ is unique means to say that is unique up to an additive constant.
If there exist just one Mather probability then $u_+$ is unique (in this sense).
In the case when there exist two points $x_1$ and $x_2$ in $ \mathbb{S}^1$ where
the supremum of the potential $V$ is attained the fixed point $u_+$ may not be unique.

Now we define, the negative Lax-Oleinik operator: for all $t>0$ and  for all continuous function $u$ on the circle, we have
$$(\mathcal{T}^-_t (u))\, (x)=$$
$$ \sup_{\gamma \in \mathbb{ X}_{t,x}} \!\Big\{  u(\gamma(0)) + \int_0^t \!\!\big[(
\dot{\gamma}(s) \log \Big(\frac{ ( \dot{\gamma}(s) +
\sqrt{\dot{\gamma}^2 (s) + 4} }{2} \Big) - \sqrt{\dot{\gamma}^2(s) +
4} + 2 - V (\gamma(s))\big] \,d s \Big\}.
$$
Note on this new definition the difference from $+$ to $-$. The space of curves we consider now is also different.
It is also known that there exists a Lipschitz function $u_-$ such
that for the same  constant $c$ as above, we have  for all $t>0$
$$\mathcal{T}^-_t (u_-) = u_-  - c \, t .$$
We say that   $u_-$ is a $(-)$-solution of the Lax-Oleinik
equation.

The $u_{+}$ solution will help to estimate the asymptotic of the left eigenvalue and the $u_{-}$ solution will help to estimate the asymptotic of the right eigenvalue of $k\,L_k+ k V_k$.

We point out that for $t$ fixed the above operator is a weak contraction. Via the discounted method is possible to approximate the scheme used to obtain $u$ by a procedure which takes advantage of another transformation which is a contraction in a complete metric space (see \cite{G1}). This is more practical for numerical applications of the theory. Another approximation scheme is given by the entropy penalized method (see \cite{GV} and \cite{GLM}).

For $k\in \mathbb{N}$ fixed the operator $ k \, L_k$
is symmetric when acting on $\mathcal{ L}^2 $ functions defined on the set $\Gamma_k\subset \mathbb S^1$. The   stationary probability of the associated Markov Chain is
the uniform measure $\pi_k$ (each point has mass $1/k)$.
When $k$ goes to infinity $\pi_k$ converges to the
Lebesgue measure on $ \mathbb{S}^1$. When the system is disturbed by $k\, V_k$ we get  new stationary probabilities $\pi_{k,V}$ with support on $\Gamma_k$ and we want to use
results of Aubry-Mather theory to estimate the large deviation properties of this family of probabilities on $\mathbb S^1$, when $k\to \infty.$

As we saw before, any weak limit of  subsequence of probabilities $\pi_{k,V}$ on $ \mathbb{S}^1=[0,1)$ is supported in the points which
attains the maximal value of $V:[0,1)\to \mathbb{R}$.
Notice that, the
supremum of
$$
\sup_{ \psi \in \mathbb{ L}^2(d\, x), \, ||\psi||_2=1}\{ \int\, V(x)
\, (\psi  (x))^2 \, d \,x\, \}=\sup \{V(x)\,|\, x \in  \mathbb{S}^1\, \},$$
is not attained on  $\mathbb{ L}^2(d\, x)$. Considering a more
general problem on the set $ \mathbb{ M} ( \mathbb{S}^1)$, the set of
probabilities on $ \mathbb{S}^1$, we have
$$
\sup_{ \nu \in \mathbb{ M} ( \mathbb{S}^1)}\{ \int\, V(x)  \, d \nu(x)\,
\}=\sup \{V(x)\,|\, x \in  \mathbb{S}^1\, \},$$ and the supremum is
attained, for example, in a delta Dirac on a point $x_0$,
where the supremum of $V$ is attained. Any measure $\nu$ which
realizes the supremum on $\mathbb{ M} ( \mathbb{S}^1)$ has support in the set
of points which attains the maximal value of $V$.
In this way the lagrangian $L$ described before appears in a
natural way.

\subsection{Large deviations for the stationary measures $\pi_{k,V}$.}
We start this subsection with same definitions.
For each $k$ and $x\in\mathbb S^1$ we denote $x_k(x)$ the closest element to $x$ on the left of $x$ in the set  $ \Gamma_k$, in fact $x_k(x)=\frac{\lfloor kx\rfloor}{k}$. Given $k$ and a  function $\varphi_k$ defined on  $ \Gamma_k$, we consider the extension $g_k$  of $\varphi_k$ to $\mathbb S^1$. This is a piecewise constant function such that in the interval $[j/k,(j+1)/k)$ is equal to $\varphi_k(j/k).$ Finally, we call $h_k$ the continuous function obtained from $g_k$ in the following way: $h_k$ is equal $g_k$ outside the intervals of the form $[\frac{j}{k} - \frac{1}{k^2} , \frac{j+1}{k} - \frac{1}{k^2}]$, $j=1,2,...,k$, and, interpolates linearly $g_k$ on these
small intervals.

When we apply the above to $\varphi_k=u_k$ the resulting $h_k$ is denoted by $z_k=z_k^V$, and when we do the same for $\varphi_k=\mu_k$, the resulting $h_k$ is called $p_{\mu_k}^V$.  In order to control the asymptotic with $k$ of  $\pi_{k,V}= u_k\, \mu_k$ we have to control the asymptotic of $z_k^V$.
We claim that $ (1/k) \, \log z_k $ is
an equicontinuous family of transformations, where $z_k$ is the
"extended continuous" to $[0,1]$.
And, we consider now  limits of  a convergent subsequences of
$z_k=z_k^V$.

\begin{lemma}

Suppose  that $u$ is a limit point of  a convergent subsequence  $(1/k_j) \,  \log
z_{k_j} $, $j \to \infty$, of $(1/k) \,  \log
z_{k} $.
Then, $u$ is a $(+)$-solution of the Lax-Oleinik equation.
\end{lemma}

\begin{proof}
We assume that $z_{k_j} \sim e^{ u \,k_j}.$ In more precise terms,  for any $x$, we have $z_{k}(x_k(x)) \sim e^{ u (x)\,k}.$
Therefore, for $t$ positive and $x$ fixed, from \eqref{*}, we have
\begin{equation*}
\begin{split}
c(L)  \, t \, + \, u(x)    =   \lim_{j \to \infty} \frac{1}{k_j}
\log ( e^{  \lambda(k_j) \, t}\,    z_{k_j} (x) ) .
\end{split}
\end{equation*}
By definitions in the begin of this subsection, we have that the expression above becomes
$$\lim_ {j \to \infty} \frac{1}{k_j}
\log\,\big[  \,( P^t_{k_j,V}  z_{k_j})(x_{k_j}
(x))\,\big] .  $$
Using again that $z_{k}(x_k(x)) \sim e^{ u (x)\,k}$, we have
$$\lim_ {j \to \infty} \frac{1}{k_j} \log\, \big[\,  (P^t_{k_j ,V} e^{ k_j \, u }) (x_{k_j} (x)) \big] =
(\mathcal{T}^+_t (u) ) \, (x).$$
Therefore, $u$ is a $(+)$-solution of the Lax-Oleinik equation
above.

\end{proof}

We point out that from the classical Aubry-Mather theory,  it follows that the fixed point $u$ for the Lax-Oleinik Operator is unique
up to an additive constant in the case the point of maximum for $V$ is unique. It follows  in this case that any convergent subsequence $(1/k_j) \,  \log
z_{k_j}^V\,\, $, $j \to \infty$,  will converge to a unique $u_{+}$. We point out that the normalization we assume for $\mu_k$ and $u_k$ (which determine $z_k$) will produce a $u_{+}$ without the ambiguity of an additive constant.

In the general case (more than one point of maximum for the potential $V$) the problem of convergence of $(1/k) \,  \log
z_{k}^V $, $k \to \infty$, is complex and is related to what is called selection of subaction. This kind of problem in other settings is analysed in \cite{AIP} and \cite{BLL}.

\medskip

One can show in a similar way that:
\begin{lemma}
Suppose  that $u^*$  is a limit point of  a convergent subsequence $(1/k_j) \,  \log
p_{k_j}^V $, $j \to \infty$, of $(1/k) \,  \log
p_{k}^V $.
Then, $u^*$ is a $(-)$-solution of the Lax-Oleinik equation.
\end{lemma}
\medskip

In the case the point of maximum for $V$ is unique one can show that any convergent subsequence $(1/k_j) \,  \log
p_{k_j}^V $, $j \to \infty$,  will converge to a unique $u^*$.

\medskip

Now, we will show that  $(1/k) \,  \log
z_{k}^V\,\, $, $k \in \mathbb{N}$, is a equicontinuous family.
\medskip

Consider now any points $x_0,x_1\in [0,1)$, a fixed positive $t\in
\mathbb{R}$, then define $\mathbb{ X}_{t,x_0,x_1}= \{\gamma(s)\in
\mathcal{AC} [0,t]\, | \, \gamma(0) = x_0, \gamma(t)=x_1\}$.

For any $x_0,x_1\in [0,1)$ and  a fixed positive $t\in \mathbb{R}$
consider the continuous functional $\phi_{t,x_0,x_1,V} : \mathbb{
X}_{t,x_0,x_1} \to {\mathbb R}$, given by
$$\phi_{t,x_0,x_1,V}
(\gamma)= \int_0^t \, (V (\gamma(s))- c(L))\, ds=\int_0^t \, V
(\gamma(s))\, ds - c(L) \, t  .$$

For a fixed $k$, when we write
$\phi_{t,x_k(x_0),x_k(x_1),V} (\gamma)$ we mean
$$\phi_{t,x_k(x_0),x_k(x_1),V} (\gamma)=
\int_0^t \, (V (x_k(\gamma(s)))- c(L))\, ds,$$
recall that $x_k(a)=\frac{\lfloor ak\rfloor}{k}$, for $a\in[0,1]$. Denote by
$\Phi_t (x_0,x_1)=\inf \{\int_0^t\, L(\gamma(s),\gamma '(s))\, ds
+ c(L)\, t\, | \, \gamma \in\mathbb{
X}_{t,x_0,x_1}\}.$ From section 3-4 in \cite{CI} it is known that
$\Phi_t (x_0,x_1)$ is Lipschitz in $ \mathbb{S}^1\times  \mathbb{S}^1$.

Given $x$ and $k$, we denote by $i(x,k)$ the natural number such
that $x_k (x) = \frac{i(x,k)}{k}.$
An important piece of information in our reasoning is
$$ \lim_{k \to \infty} \pfrac{1}{k} \log
(e^{t\,(\,  k \, \,L_k + k\, \,V_k \,-\,\lambda(k))})_{i(x_0,k)\,
i(x_1,k)} $$
$$ =\lim_{k\to\infty} \frac{1}{k} \log
\mathbb{E}_{X_k(0)=\frac{i(x_0,k)}{k},X_k(t)=\frac{i(x_1,k)}{k}}^k
[e^{k\,
\phi_{t,x_k(x_0),x_k(x_1),V}\, (.) } ] $$$$=
\sup_{\gamma \in \mathbb{ X}_{t,x_0,x_1}} \{ \phi_{t,x_0,x_1,V} (\gamma) -
I_t(\gamma)\}.$$
The last equality is from Varadhan's Integral Lemma. Using the definition of $\phi_{t,x_0,x_1,V}$ and of $I_t$, see \eqref{funcional}, we get
\begin{equation*}
\begin{split}
&\sup_{\gamma \in \mathbb{ X}_{t,x_0,x_1}} \{ \phi_{t,x_0,x_1,V} (\gamma) -
I_t(\gamma)\}\\
&=\sup_{\gamma \in \mathbb{X}_{t,x_0,x_1}} \Big\{\int_0^t V (\gamma(s)) ds - c(L) \, t\\
&\qquad\qquad- \int_0^t \big[ \dot{\gamma}(s) \log \Big(\frac{ \dot{\gamma}(s) +
\sqrt{\dot{\gamma}^2 (s) + 4} }{2} \Big) - \sqrt{\dot{\gamma}^2(s) +
4} + 2 \big]\, d s \Big\} \\
&=\sup_{\gamma \in \mathbb{
X}_{t,x_0,x_1}} \Big\{-\, \int_0^t L^V (\gamma(s), \gamma' (s))\,  ds
\,- c(L) \, t \Big\}\\
&=- \inf_{\gamma \in \mathbb{ X}_{t,x_0,x_1}}
\Big\{\, \int_0^t L^V (\gamma(s), \gamma' (s))\,  ds \, + \,  c(L) \, t
\Big\}= - \Phi_t (x_0,x_1).
\end{split}
\end{equation*}
The convergence is uniform on $k$, for any $x_0,x_1$. And, the definition of $L^V$ is on \eqref{L}.

\medskip

\begin{lemma}  The family $\pfrac{1}{k} \log
z_k^V$ is equicontinuous in $k\in \mathbb{N}$. Therefore, there  exists a subsequence of  $\pfrac{1}{k} \log
z_k^V$ converging to a certain Lipschitz function $u$. In the case the maximum of $V$ is attained in a unique point, then
$u$ is unique up to an additive constant.
\end{lemma}

\begin{proof}

 Given $x$ and $y$, and a positive fixed $t$ we
have
$$ \pfrac{1}{k} \log z_k ( x_k (x))-\pfrac{1}{k} \log z_k( x_k (y))=$$
$$\pfrac{1}{k} \log \frac{\sum_{j=0}^{k-1}  \,(e^{t\,(\,  k \, \,L_k + k V_k)})_{i(x,k)\, j} z_j }{
\sum_{j=0}^{k-1} \,(e^{t\, (\, k\,L_k + k V_k)})_{i(y,k)\,
j}z_j }\leq
$$
$$\pfrac{1}{k} \log \, \Big(\, \sup_{j=\{0,1,2,..k-1\}} \,\Big\{\,\, \frac{
\,(e^{t\,(\,  k \, \,L_k + k V_k)})_{i(x,k)\, j}}{ \,(e^{t\,
(\, k\,L_k + k V_k)})_{i(y,k)\, j} }\,\, \Big\}\,\Big )
$$

For each $k$ the above supremum is attained at a certain $j_k$.
Consider a convergent subsequence $\frac{j_k}{k}$ to a certain
$z$, where $k\to \infty$. That is, there exists $z$ such that $i(z,k)=j_k$ for all $k$.

Therefore, for each $k$ and $t$ fixed
$$ \pfrac{1}{k} \log z_k ( x_k (x))-\pfrac{1}{k} \log z_k( x_k
(y))\leq \pfrac{1}{k} \log \,\, \frac{ \,(e^{t\,(\, k\,
\,L_k + k V_k)})_{i(x,k)\, j_k}}{ \,(e^{t\, (\, k\,L_k + k
V_k)})_{i(y,k)\, j_k} }$$
$$=\pfrac{1}{k} \log \,\, \frac{ \,(e^{t\,(\, k \, \,L_k + k
V_k)})_{i(x,k)\, i(z,k)}}{ \,(e^{t\, (\, k\,L_k + k
V_k)})_{i(y,k)\, i(z,k)} } .$$
Taking $k$ large, we have, for $t$ fixed that
$$ \pfrac{1}{k} \log z_k ( x)-\pfrac{1}{k} \log z_k(
y)\leq \Phi_t (y,z) - \Phi_t (x,z).$$
The Peierls barrier is defined as
$$  h(y,x)= \varliminf_{t\to \infty} \Phi_t(y,x) .$$
Taking a subsequence $t_r\to \infty$ such $h(y,z)=\varliminf_{r\to
\infty} \Phi_{t_r} (x,z)$, one can easily shows that for large $k$
$$ \pfrac{1}{k} \log z_k ( x)-\pfrac{1}{k} \log z_k(
y)\leq h (y,z) - h (x,z).$$
The Peierls barrier satisfies $ h(y,z)-h(x,z)\leq \Phi (y,x)\leq A
\, |x-y|$, where $A$ is constant and $\Phi$ is the Ma\~ne
potential (see 3-7.1\, item 1. in \cite{CI}).
Therefore, the family is equicontinuous. For each $k$ fixed there
is always a value $z_k(x)$ above $1$ and one below $1$.

The conclusion is that there exists a subsequence of  $\frac{1}{k}
\log z_k $ converging to a certain $u$.
The uniqueness of the limit follows from the uniqueness of $u$

\end{proof}

A similar result is true for  the family $\frac{1}{k} \,  \log p_{\mu_k}^V$, remember that $p_{\mu_k}^V$ is obtained through of $\mu_k$.
Taking a convergent subsequence, we denote by $u^*$ the limit. This subsequence can be considered as a subsequence of the one we already got convergence
for  $\frac{1}{k}\,  \log z_k^V.$ In this case we got an $u=u: \mathbb{S}^1 \to \mathbb{R}$ and a $u^*: \mathbb{S}^1 \to \mathbb{R}$, which are limits of the corresponding subsequences.

\medskip

Now we want to analyse large deviations of the measure $\pi_{k,V}$.

\medskip

 \begin{theorem} A large deviation principle for  the sequence of  measures $\{\pi_{k,V}\}_k$ is true and the deviation rate function $I^V$ is $I^V(x)=  u (x) + u^{*} (x)$. In other words, given an interval $F=[c,d]$,
$$ \lim_{k\to \infty} \frac{1}{k} \,\log \pi_{k,V}\,[\,F\,]\,=
-\, \inf \{ I(x) \, | \, x \in F\}.$$
 \end{theorem}

\begin{proof}

Suppose the maximum of $V$ is unique. Then, we get  $z_{k}(x_k(x)) \sim e^{ u_{+} (x)\,k}$ and $p_{\mu_k}^V(x_k(x)) \sim e^{ u_{-} (x)\,k}$
What is the explicit expression for $I^V$?
Remember that $u^{+}$ satisfies $\mathcal{T}^+_t (u_+) = u_+  + c \, t $ and $u^{-}$ satisfies $\mathcal{T}^+_t (u_{-}) = u_{-} + c \, t $.
Here, $u$ is one of the $u_{+}$ and $u^*$ is one of the $u_{-}$. As we said before they were determined by the normalization. The functions
$u_{+}$ and $u_{-}$ are weak KAM solutions.

We denote $I^V(x)=  u (x) + u^{*} (x).$ The function $I^V$ is continuous (not necessarily differentiable in all $\mathbb{S}^1$) and well defined.
Notice that $\pi_{k,V}(j/k) = (z_k^V)_j \,
(p_{\mu_k}^V)_j .$
We have to estimate
$$\pi_{k,V}\,[\,F\,]\, = \sum_{j/k \in F} p_{mu_k}(j/k) z_k(j/k)\sim \sum_{j/k \in F}e^{k (u_{-} ( x_k(j/k))+ u_{+} ( x_k(j/k))} .$$
Then, from Laplace method it follows that $I^V(x)$ is the deviation function.

\end{proof}

\section{Entropy of $V$.}\label{sec4}

\subsection{Review of the basic properties of the entropy for continuous time Gibbs states}

In \cite{LNT} it is consider the Thermodynamic Formalism for continuous time Markov Chains taking values in the Bernoulli space. The authors consider
a certain a priori potential
$$A:\{1,2,...,k\}^\mathbb{N}\to \mathbb{R}$$ and an associated discrete Ruelle operator ${\mathcal  L}_A$.

Via the infinitesimal generator $L = {\mathcal  L}_A-I$ is defined an a priori probability over the Skorohod space

In \cite{LNT} it is consider a potential $V:\{1,2,...,k\}^\mathbb{N}\to \mathbb{R}$ and the continuous time Gibbs state associated to $V$. This generalizes what is know for the discrete time setting of Thermodynamic Formalism (see \cite{PP}). In this formalism the properties of the Ruelle operator  ${\mathcal  L}_A$  are used to assure the existence of eigenfunctions, eigenprobabilities, etc... The eigenfunction is used to normalize the continuous time semigroup operator in order to get an stochastic semigroup (and a new continuous time Markov chain which is called Gibbs state for $V$). The main technical difficulties arise from the fact that the state space of this continuous time Markov Chain is not finite (not even countable). \cite{Ki1} is a nice reference for the general setting of Large Deviations
in continuous time.

By the other hand, in \cite{BEL} the authors considered continuous time Gibbs states in a much more simple situation where the state space is finite. They consider an infinitesimal generator which is a $k$ by $k$  matrix $L$ and a potential $V$ of the form $V:\{1,2,...,k\}\to \mathbb{R}$. This is
more close to the setting we consider here with $k$ fixed.

In the present setting, and according to the notation of last section,  the semigroup $e^{t\, (k\, L_k + k\,V_k - \lambda(k))}, t>0,$  defines what we call the continuous time Markov chain associated to $k\, V_k$. The vector $\pi_{k,V}=(\pi_{k,V}^1,...,\pi_{k,V}^k)$, such that $\pi_{k,V}^j=\,u_k^j\, \mu_k^j\, \,$, $j=1,2,..,k$, is stationary for such Markov Chain.

Notice that the semigroup $e^{t\, (k\, L_k + k\,V_k}), t>0,$ is not stochastic and the procedure of getting an stochastic semigroup from this requires a normalization via the eigenfunction and eigenvalue.

If one consider a potential $A:\{1,2,...,k\}^\mathbb{N}\to \mathbb{R}$ which depends on the two first coordinates and a potential
$V:\{1,2,...,k\}^\mathbb{N}\to \mathbb{R}$ which depends on the first coordinate one can see that "basically" the  results of \cite{LNT} are an extension of the ones in \cite{BEL}.

In Section 4 in \cite{LNT} it is consider a potential $V:\{1,2,...,k\}^\mathbb{N}\to \mathbb{R}$ and introduced for the associated Gibbs continuous time Markov Chain, for each $T>0$, the concept of entropy $H_T$. Finally, one can take the limit on $T$ in order to obtain an entropy $H$ for the continuous time Gibbs state associated to such $V$. We would like here to compute for each $k$ the expression of the entropy $H(k)$ of the
Gibbs state for $k V_k$.
Later we want to estimate the limit $H(k)$, when $k\to \infty$.

Notice that for fixed $k$ our setting here is a particular case (much more simpler) that the one where the continuous time Markov Chain has the state space $\{1,2,...,k\}^\mathbb{N}$. However, the matrix $L_k$ we consider here assume some zero values and this was not explicitly considered in \cite{LNT}.
This will be no big problem because  the use of the discrete time Ruelle operator in \cite{LNT} was mainly for showing the existence of eigenfunctions and eigenvalues. Here the existence of eigenfunctions and eigenvalues follows from trivial arguments due to the fact that the
operators are defined in finite dimensional vector spaces.

A different approach to entropy on the continuous time Gibbs setting (not using the Ruelle operator)  is presented in \cite{Leav}.
We point out that \cite{BEL} does not consider the concept of entropy.
We will show below that for the purpose of computation of the entropy for the present setting the reasoning of \cite{LNT} can be described in more general terms without mention the  Ruelle operator ${\mathcal  L}_A$.

No we will briefly describe for the reader the computation of entropy in \cite{LNT}.
Given a certain a priori Lipschitz potential
$$A_k:\{1,2,...,k\}^\mathbb{N}\to \mathbb{R}$$ consider the associated discrete Ruelle operator ${\mathcal  L}_{A_k}$.

Via the infinitesimal generator $\tilde{L}_k = {\mathcal  L}_{A_k}-I$, for each $k$, we define an a priori probability  Markov Chain. Consider now a potential $\tilde{V}_k:\{1,2,...,k\}^\mathbb{N}\to \mathbb{R}$ and the associated Gibbs continuous time Markov Chain. We denote by $\mu^{k}$ the stationary vector for such chain. We denote by $P_{\mu^{k}}$ the probability over the Skorohod space $D$ obtained from initial probability  $\mu^{k}$ and the a priori Markov Chain (which will define a Markov Process which is not stationary). We also consider $\tilde{P}^{\tilde{V}_k}_{\mu^k}$ the probability on $D$ induced by the continuous time Gibbs state associated to $V$ and the initial measure $\mu^k$.

According to Section 4 in \cite{LNT}, for a fixed $T\geq 0$,  the relative entropy is
\begin{equation}\label{entropy}
 H_T(\tilde{P}^{\tilde{V}_k}_{\mu^k}\vert P_{\mu^{k}})\,=-\,\int_{ D}
\log\Bigg(\frac{\mbox{d}\tilde{ P}^{\tilde{V}_k}_{\mu^k}}{\mbox{d} P_{\mu^{k}}}\Big|_{ \mathcal{F}_T}\Bigg)(\omega)\,
\mbox{d}\tilde{ P}^{\tilde{V}_k}_{\mu^k} (\omega)\,.
\end{equation}

In the above $\mu_k$ is a probability fixed on the state space and $\mathcal{F}_T$ is the usual sigma algebra up to  time $T$. Moreover, $D$ is the Skorohod space.

The entropy of the stationary Gibbs state $\tilde{ P}^{\tilde{V}_k}_{\mu^k}$ is

\begin{equation*}
H(\tilde{P}^{\tilde{V}_k}_{\mu^k}\vert P_{\mu^{k}})\,=\,\lim_{T\to \infty} \frac{1}{T}  H_T(\tilde{P}^{\tilde{V}_k}_{\mu^k}\vert P_{\mu^{k}}).
\end{equation*}

The main issue here is to apply the above to $k\, V_k$ and not $\tilde{V_k}.$ In order to compute the entropy in our setting  we have to show that the expression above can be generalized and described not mentioning the a priori potential $A$. This will be explained in the next section.

\subsection{Gibbs state in a general setting}
The goal of this subsection is improve the results of the Sections 3 and  4 of the paper \cite{LNT}.  In order to do this we will consider a continuous time  Markov Chain $\{X_t, t\geq 0\}$ with state space $E$ and with
 infinitesimal generator given by
 \begin{equation*}
\begin{split}
 L(f)(x)=\sum_{y\in E}p(x,y)\big[f(y)-f(x)\big],\\
 \end{split}
\end{equation*}
where $p(x,y)$ is the rate jump from $x$ to $y$. Notice that maybe $\sum_{y\in E}p(x,y)\neq 1$. For example, if the state space $E$ is $\{1,...,k\}^{\mathbb{N}}$ and $L=\mathcal L_A-I$, as in \cite{LNT}, we have that $p(x,y)=\mathbf{1}_{\sigma(y)=x}e^{A(y)}$, or if $L=L^V$, also in \cite{LNT}, $p(x,y)$ is equal to $\gamma_V(x)\mathbf{1}_{\sigma(y)=x}e^{B_{V}(y)}$.

\medskip

As we will see by considering this general $p$ one can get more general results.

\medskip

\begin{proposition}
 Suppose $L$ is an infinitesimal generator as above and $V:E\to \mathbb{R}$ is a function  such that  there exists an associated eigenfunction $F_V:E\to (0,\infty)$ and  eigenvalue $\lambda_V$ for $L+V$. That is,  we have that $(L+V)F_V=\lambda_V\, F_V$. Then, by a procedure of normalization, we can get a new continuous time Markov Chain, {\bf called the continuous time Gibbs state for $V,$} which is the process $\{Y^V_T,\,T\geq 0\}$, having the infinitesimal generator acting on bounded mensurable functions $f:E\to\mathbb R$ given by
\begin{equation}\label{LV}
 L^{V}(f)(x)= \sum_{y\in E}\frac{p(x,y)F_V(y)}{F_V(x)}\big[f(y)-f(x)\big]\,.
\end{equation}
\end{proposition}
\begin{proof}
To obtain this infinitesimal generator we can follow without any change from the beginning of the proof of the Proposition 7 in Section 3 of \cite{LNT} until we get  the equality (11). After the equation (11) we use the fact that $p(x,y)$ is equal to $\mathbf{1}_{\sigma(y)=x}e^{A(y)}$. Then, in the present setting we just have  to start from the equation (11). Notice that  the infinitesimal generator $L^V(f)(x)$ can be written as
\begin{equation*}
\begin{split}
&\frac{L(F_Vf)(x)}{F_V(x)}+ (V(x)-\lambda_V)f(x)\\&=\sum_{y\in E}\frac{p(x,y)}{F_V(x)}\big[F_V(y)f(y)-F_V(x)f(x)\big]+ (V(x)-\lambda_V)f(x)\\
&=\sum_{y\in E}\frac{p(x,y)F_V(y)}{F_V(x)}f(y)+ ([\sum_{y\in E}p(x,y)]+V(x)-\lambda_V)f(x)\,.\\
\end{split}
\end{equation*}
Using the fact that $F_V$ and $\lambda_V$ are, respectively, the eigenfunction and eigenvalue, we get that
the expression \eqref{LV} defines and infinitesimal generator for a continuous time Markov Chain
\end{proof}
\medskip

Now, rewriting \eqref{LV} as
\begin{equation*}
 L^{V}(f)(x)= \sum_{y\in E}p(x,y)\,e^{\log F_V(y)-\log F_V(x)}\big[f(y)-f(x)\big]\,,
\end{equation*}
 we can see that the process $\{Y_T^V, T\geq 0\}$ is a perturbation of the original process $\{X_t, t\geq 0\}$. This perturbation is given by the function $\log F_V$, where $F_V$ is the eigenfunction of $L+V$, in the sense of the Appendix 1.7 of \cite{KL}, page 337.

 \medskip
 Now we  will introduce a natural concept of entropy for this more general setting describe by the general function $p$.
\medskip

Denote by $\mathbb P_\mu$ the  probability on the Skorohod space $D:=D([0, T], E)$ induced by $\{X_t, t\geq 0\}$  and the initial measure $\mu$.
And, denote by $\mathbb P^V_\mu$ the  probability on $D$ induced by $\{Y_T^V, T\geq 0\}$  and the initial measure $\mu$.
By \cite{KL}, page 336, the Radon-Nikodym derivative $\frac{d\mathbb P^V_\mu}{d \mathbb P_\mu}$ is
\begin{equation*}
\begin{split}
&\exp\Big\{\log F_V(X_T)-\log F_V(X_0)-\int_0^T\frac{L(F_V)(X_s)}{F_V(X_s)}\,ds\Big\}\\
=&\exp\Big\{\log \frac{F_V(X_T)}{F_V(X_0)}+\int_0^T(V(X_s)-\lambda_V)\,ds\Big\}\\
=& \frac{F_V(X_T)}{F_V(X_0)}\exp\Big\{\int_0^T(V(X_s)-\lambda_V)\,ds\Big\}.\\
\end{split}
\end{equation*}

Thus, we obtain the expression:
\begin{equation*}
\begin{split}
&\log\Big(\frac{d\mathbb P^V_\mu}{d \mathbb P_\mu}\Big)=\int_0^T(V(X_s)-\lambda_V)\,ds+\log F_V(X_T)-\log F_V(X_0).\\
\end{split}
\end{equation*}
which is more sharp that the expression (17) on page 13 of \cite{LNT}.
To compare them,  we take on (17)
$\tilde{\gamma}=1-V+\lambda_V$, then we obtain the first term. To obtain the second one, we need to observe that the second term in (17), in \cite{LNT}, can be written as a telescopic sum.
\medskip

Now for a fixed $k$ we will explain how to get the value of the entropy of the corresponding Gibbs state for $k\, V_k: \Gamma_k \to\mathbb{R}$.

In the general setting of last theorem consider $E =  \Gamma_k=\{0,1/k,2/k,..,(k-1)/k\}$, and, for $i/k,j/k\in \Gamma_k$, we have

a) $p(i/k,j/k)= k$, if $j=i+1$ or $j=i-1$,

b) $p(i/k,j/k)=0,$ in the other cases.

The existence of eigenfunction $F_k$ and eigenvalue $\lambda_k$  for $k L_k + k V_k$ follows from the continuous time Perron's Theorem described before. The associated continuous time Gibbs Markov Chain has a initial stationary vector which will be denoted by $\pi_k$.

Now we have to integrate concerning $\mathbb P_{\pi_{k,V}}^{kV_k}$ for $T$ fixed the function
$$\int_0^T(k\,V_k(X_s)-\lambda_k)\,ds+\log F_k(X_T)-\log F_k(X_0).$$

As the probability that we considered  on the Skorohod space is stationary and ergodic this integration results in
$$\int k V_k d \pi_{k,V} - \lambda_k.$$
Thus, the entropy $H(\mathbb P_{\pi_{k,V} }^{kV_k}\vert \mathbb P_{\pi_{k,V} })=\int k V_k d \pi_{k,V}  - \lambda_k$.
We point out that for a fixed $k$  this number is computable from the linear problem associated to the continuous time Perron's operator.
Now in order to find the limit entropy associated to $V$ we need to take the limit on $k$ of the above expression.

Here, we  assume that the Mather measure is a Dirac Delta probability on $x_0.$
Remember that $\lim_{k\to\infty}\frac{1}{k}\lambda(k) =c(L)= V(x_0).$ Moreover, $\pi_{k,V}  \to \delta_{x_0}$, when $k\to \infty$.
Therefore,
$$H(V)=\lim_{k\to\infty}\frac{1}{k}H(\mathbb P_{\pi_{k,V} }^{kV_k}\vert \mathbb P_{\pi_{k,V} })=
\lim_{k\to\infty}\int  V_k \,d \pi_{k,V}  - \lim_{k\to\infty}\frac{1}{k}\lambda_k=V(x_0)-c(L)=0.$$
The limit entropy in this case is zero.


\begin{thebibliography}{99}
\footnotesize




\bibitem{A1} N. Anantharaman, Counting geodesics which are optimal in
homology, Ergodic  Theory and Dynamical Systems, Vol 23 Issue 2
(2003).

\bibitem{A2} N. Anantharaman, On the zero-temperature or vanishing viscosity limit
for Markov processes arising from Lagrangian dynamics, J. Eur. Math.
Soc. 6 no. 2, 207–276 (2004).

\bibitem{AIP}
N. Anantharaman, Nalini, R. Iturriaga, P. Padilha and H. Sanchez-Morgado, Physical solutions of the Hamilton-Jacobi equation. Discrete Contin. Dyn. Syst. Ser. B 5, no. 3, 513-528 (2005).

\bibitem{BEL} A. Baraviera, R. Exel and A. Lopes, A Ruelle Operator for continuous time Markov
chains,   S\~ao Paulo Journal of Mathematical Sciences. vol 4 n. 1, pp 1-16 (2010).



\bibitem{BLL} A. Baraviera, R. Leplaideur and A. O. Lopes, Selection of ground states in the zero temperature limit for a one-parameter family of potentials, \emph{SIAM Journal on Applied Dynamical Systems}, Vol. 11, n 1, 243-260 (2012).

\bibitem{BG}
A. Biryuk and D. A. Gomes, An introduction to Aubry-Mather theory. Sao
Paulo Journal of Mathematical Sciences, 4 (1), 17--63, 2010


\bibitem{Car} M. J. Carneiro,
On minimizing measures of the action of autonomous Lagrangians,
Nonlinearity 8 (1995), no. 6, 1077–-1085.

\bibitem{CI} G. Contreras and R. Iturriaga, Global minimizers of
autonomous lagrangians, CIMAT, (2000) (see homepage of  G.
Contreras in CIMAT).


\bibitem{DZ}  A. Dembo and O. Zeitouni, Large Deviations techniques, Springer Verlag.



\bibitem{Ellis} R. Ellis,
Entropy, Large Deviations, and Statistical Mechanics,
Springer Verlag


\bibitem{EK} S. Ethier and T. Kurtz, Markov Processes, John Wiley, (1986).



\bibitem{Fath}
A. Fathi, Th\'eor\`eme KAM faible et th\'eorie de Mather sur les syst\`emes lagrangiens,
Comptes Rendus de l'Acad\'emie des Sciences, S\'erie I, Math\'ematique Vol 324   1043-1046, 1997.


\bibitem{Fat}
A. Fathi, Weak KAM theorem in Lagrangian Dynamics, Lecture Notes, Pisa (2005)

\bibitem{FW} M. I. Freidlin , A. D. Wentzel, Random Perturbations of Dynamical Systems, Springer, (1991).

\bibitem{G1}
D. A. Gomes, Viscosity solution methods and discrete Aubry–Mather problem,
Discrete Contin. Dyn. Syst. 13(1) (2005) 103–-116.

\bibitem{GV}
D. A. Gomes and E. Valdinoci, Entropy penalization methods for Hamilton–Jacobi
equations, Adv. Math. 215(1) (2007) 94–-152.


\bibitem{Kac} M. Kac, Integration in Function spaces and some of its
applications, Acad Naz dei Lincei Scuola Superiore Normale
Superiore, Piza, Italy (1980).


\bibitem{Ki1} Y. Kifer, Large Deviations in Dynamical Systems and Stochastic processes, TAMS, Vol 321, N.2, 505--524 (1990)


\bibitem{KL} C. Landim and C. Kipnis,  Scaling limits of interacting
  particle systems. Grundlehren der Mathematischen Wissenschaften, 320.
  Springer-Verlag, Berlin (1999).

\bibitem{Leav}
V. Lecomte, C. Appert-Rolland and F. van Wijland, Thermodynamic formalism for systems with Markov dynamics. J. Stat. Phys. 127 (2007), no. 1,
51-106

\bibitem{GLM} D. Gomes, A. Lopes and J. Mohr,
The Mather measure and a Large Deviation Principle for the Entropy Penalized Method, Communications in Contemporary Mathematics, Vol 13, issue 2, 235--268 (2011)

\bibitem{PP}
W. Parry and M.  Pollicott, Zeta functions and the periodic
orbit structure of hyperbolic dynamics, \emph{Ast\'erisque}
Vol {187-188} 1990

\bibitem{LNT} A. O. Lopes, A. Neumann and Ph. Thieullen
A thermodynamic formalism for continuous time Markov chains with values on the Bernoulli Space: entropy, pressure and large deviations,
Journ. of Statist. Phys.  Volume 152, Issue 5, Page 894-933 (2013).

\bibitem{A} A. Neumann: Large Deviations Principle for  the Exclusion Process with Slow Bonds, PhD Thesis at IMPA (2011).


\bibitem{N}  J. B. Norris,  Markov Chains, Cambridge Press

\bibitem{OV} E. Olivieri, M.E. Vares: Large deviations and Metastability.  Cambridge Universtiy Press, Cambridge (1998).


\bibitem{S} D. W. Strook, An introduction to Large Deviations, Springer,
(1984).


\bibitem{SK} A. Skhorokhod, Studies in the theory of Random Processes,
Dover.

\bibitem{We} A. D. Wentzell, Limit Theorems on Large Deviations for Markov Stochastic Porcesses, Kluwer, (1990)



\end{thebibliography}
\end{document}